\documentclass[11pt]{scrartcl}
\usepackage[margin=1.0in]{geometry}
\usepackage[dvips]{graphicx}
\usepackage{color,amsmath,latexsym,amssymb,amsthm}
\usepackage{mathtools}
\usepackage{xspace}
\usepackage{pgf,tikz}
\usepackage{rotating}
\usepackage{multirow}
\usepackage{algorithmic}
\usepackage{framed}
\usepackage{subfig}
\usepackage{setspace}
\usepackage{appendix}
\usepackage{ifthen}
\usepackage{mathtools}
\usepackage{mathrsfs}
\usepackage{hyperref}
\usepackage{cleveref}
\usepackage{enumitem}
\usepackage{algorithm, algorithmic}
\usepackage{soul}

\newtheorem{definition}{Definition}
\numberwithin{definition}{section}
\numberwithin{equation}{section}
\newtheorem{theorem}[definition]{Theorem}
\newtheorem{lemma}[definition]{Lemma}

\theoremstyle{definition}
\newtheorem{example}[definition]{Example}
\newtheorem{assumption}[definition]{Assumption}
\newtheorem{remark}[definition]{Remark}

\crefdefaultlabelformat{#2\textup{#1}#3}
\Crefname{assumption}{Assumption}{Assumptions}

\newlist{enumthm}{enumerate}{1}
\setlist[enumthm]{label=\textup{\arabic*.},ref=\theassumption.\textup{\arabic*}}
\crefalias{enumthmi}{assumption}

\newcommand{\bsp}{\Lambda}

\newcommand{\coperator}{B}
\newcommand{\dualp}[3][]{\langle #2, #3 \rangle_{{#1}^*\!, #1}}
\newcommand{\dualpHzeroone}[3][]
{\langle #2, #3 \rangle_{H^{-1}(\domain), H_0^1(\domain)}}
\newcommand{\spL}[2]{\mathscr{L}(#1, #2)}
\newcommand{\functionArgument}[1]{\ifthenelse{\equal{#1}{}}  
	{}
	{({#1})}
}

\newcommand{\norm}[2][2]{\|#2\|_{#1}}
\newcommand{\rdc}{\kappa}
\newcommand{\rrhs}{b}
\newcommand{\inner}[3][]{( #2, #3 )_{#1}}
\newcommand{\obj}{f}
\newcommand{\pobj}{F}
\newcommand{\pttobj}{J}
\newcommand{\rpobj}{\pobj}
\newcommand{\dist}[2]{\mathrm{dist}\functionArgument{#1,#2}}
\newcommand{\deviation}[2]{\mathbb{D}\functionArgument{#1,#2}}
\newcommand{\csp}{U}
\newcommand{\adcsp}{\csp_\text{ad}}
\newcommand{\hsp}{H}
\newcommand{\domain}{D}
\newcommand{\ub}{\mathfrak{u}}

\newcommand{\real}{\mathbb{R}}
\newcommand{\embedding}{\xhookrightarrow{}}
\newcommand{\du}{\ensuremath{\mathrm{d}}}
\newcommand{\Du}{\ensuremath{\mathrm{D}}}
\newcommand{\wpone}{w.p.~$1$}
\newcommand{\frechet}{Fr\'echet}
\newcommand{\hoelder}{H\"older}
\newcommand{\friedrichs}{Friedrichs}
\newcommand{\maxo}[1]{(#1)_+}
\newcommand{\dom}[1]{\mathrm{dom}(#1)}
\newcommand{\wto}{\rightharpoonup}

\makeatletter
\DeclareOldFontCommand{\sc}{\normalfont\scshape}{\@nomath\sc}
\makeatother

\newcommand{\funArg}[1]{\ifthenelse{\equal{#1}{}}  
	{}
	{[#1]}
}

\newcommand{\AVaR}[2][]{\mathrm{AVaR}_\beta\funArg{#2}}
\newcommand{\mcAVaR}[2][]{\widehat{\mathrm{AVaR}}_{\beta,N}\funArg{#2}}

\renewcommand{\abstract}[1]
{
	{\small
		\textbf{Abstract.} {#1}
		\\
	}
}

\newcommand{\keywords}[1]
{
	{\small
		\textbf{Key words.} {#1}
		\\
	}
}

\newcommand{\amssubject}[1]
{
	{\small
		\textbf{AMS subject classifications.} {#1}
	}
}

\title{Consistency of sample-based stationary points 
	for infinite-dimensional stochastic optimization}

\author{Johannes Milz%
        \thanks{H.\ Milton Stewart School of Industrial and Systems Engineering, Georgia Institute of Technology, Atlanta, GA 30332,     USA        \texttt{johannes.milz@isye.gatech.edu}.}}
        
\date{June 19, 2023}

\begin{document}

\maketitle

\abstract{%
We consider stochastic optimization problems 
with possibly nonsmooth integrands posed in  Banach spaces and 
approximate these stochastic programs via a 
sample-based approaches. 
We establish the consistency 
of approximate Clarke stationary points of the sample-based approximations.
Our framework is applied to risk-averse semilinear PDE-constrained
optimization
using the average value-at-risk and to  risk-neutral bilinear
PDE-constrained
optimization.
}

\par
\par
\amssubject{%
	 65C05, 90C15, 35R60, 90C48, 90C30, 60H25, 49M41, 35Q93
}

\section{Introduction}
\label{sec:intro}
Infinite-dimensional  optimization problems
arise in a plethora of research fields such as dynamic programming
\cite{Langen1981}, statistical estimation \cite{Gine2016},  
feedback stabilization of dynamical systems \cite{Kunisch2020}, and
optimization problems governed by partial 
differential equations (PDEs)  \cite{Kouri2018a}.
PDE-constrained optimization  is an active research field with 
a focus on modeling, analyzing and solving  complex optimization
problems with PDE constraints. For example, 
numerous applications in the field of renewable  and sustainable energy yield
challenging PDE-constrained optimization problems,
such as wind plant layout optimization \cite{King2017}, 
wind turbine blade planform design \cite{Allen2022}, 
and tidal stream array optimization \cite{Goss2021}. 
Parameters in PDEs may be uncertain, such as 
diffusion coefficients and boundary
conditions, and as such can have a significant
influence on the simulation output. 
Using parameter-dependent 
PDEs for decision making naturally
leads to decision making under uncertainty. 
A common approach in optimization
under uncertainty models uncertain parameters as random 
variables with known
probability distribution. Using this stochastic programming approach, 
we can formulate an optimization problem having solutions performing best
on average; its  objective function is 
the expected value of a parameterized cost function, the integrand. 
Such optimization problems are nowadays often referred to as
risk-neutral problems to distinguish them from risk-averse
formulations. Since high-dimensional integrals
cannot be accurately evaluated and solving PDEs requires 
numerical computations, 
we approximate the expectations using a Monte Carlo sample-based 
approximation,
the sample average approximation (SAA) approach 
\cite{Kleywegt2002,Shapiro2003,Shapiro2007}, 
and other schemes  approximating the true probability distribution
via weakly convergent probability distributions.
In this manuscript, we demonstrate the  consistency of 
Clarke-(C-)stationary points
of the approximated problems
towards the risk-neutral problem's set of C-stationary points.

Motivated by infinite-dimensional optimization problems governed by
complex physical systems under uncertainty, 
we consider the risk-neutral composite  optimization problem
\begin{align}
	\label{eq:ocp}
	\min_{x \in X}\,  \int_{\Xi}\, \pobj_{\xi}(x) \, \du \mathbb{P}(\xi)
	+ \psi(x),
\end{align}
where $X$ is a reflexive Banach space, 
$\psi \colon X \to (-\infty,\infty]$ is a proper, convex, lower semicontinuous
function having the so-called Kadec property, 
$\mathbb{P}$ is a probability measure on the complete, 
separable metric space $\Xi$, and
$\pobj_{\xi}$ is a locally Lipschitz continuous integrand with
integrable Lipschitz constants. 

In this manuscript, we consider two types of approximations of the
expectations in \eqref{eq:ocp}. While our main focus
is on analyzing the SAA approach
\cite{Kleywegt2002,Shapiro2003}, we also consider
approximations of $\mathbb{P}$ via weakly converging probability
measures \cite{Langen1981,Feinberg2022}.
To introduce, the 
SAA approach, 
let $\xi^1$, $\xi^2, \ldots$ be independent identically distributed
$\Xi$-valued random elements such that each $\xi^i$ has
distribution $\mathbb{P}$.
The SAA problem with sample size  $N \in \mathbb{N}$ 
corresponding
to \eqref{eq:ocp} is given by
\begin{align}
	\label{eq:saa}
	\min_{x \in X}\,  \frac{1}{N} \sum_{i=1}^N \pobj_{\xi^i}(x) + \psi(x).
\end{align}

The principle contributions of this manuscript are twofold.
\begin{enumerate}[leftmargin=*]
	\item[(i)] 
	We demonstrate the consistency of approximate
	C-stationary points of the SAA problem \eqref{eq:saa} for a class of
	risk-neutral nonlinear optimization problems \eqref{eq:ocp}
	by adapting and extending a framework developed in \cite{Milz2022}, 
	where the consistency of SAA optimal values and SAA solutions 
	has been demonstrated. 
	More specifically, our analysis relies on the identity
	\begin{align}
		\label{eq:intro:subdifferentialidentity}
		\partial_C \pobj_{\xi}(x) = \coperator  M_\xi(x)
		\quad \text{for all} \quad (x,\xi) \in X \times \Xi,
	\end{align} 
	where $\partial_C \pobj_{\xi}(x) $ is Clarke's subdifferential
	of $\pobj_{\xi}$ at $x$,
	$V$ is a Banach space, $\coperator  \colon V \to X^*$
	is linear and compact, and 
	$M_\xi \colon X   \rightrightarrows V$ is a multifunction
	for each $\xi \in \Xi$.
	Moreover, using proof techniques developed in \cite{Milz2022a}, 
	we establish consistency statements for the significantly larger class of
	functions $\psi$ than those considered in
	\cite{Milz2022}, so-called functions having the Kadec property
	\cite{Borwein1994}.

	\item[(ii)]  We apply our framework to 
	risk-averse semilinear PDE-constrained optimization problems
	using the average value-at-risk 
	and to steady-state elliptic 
	bilinear optimal control problems under uncertainty. 
	While the literature on risk-neutral
	and risk-averse semilinear PDE-constrained
	optimization is extensive 
	\cite{Geiersbach2020,Kouri2016,Kouri2020,Milz2022b}, 
	bilinear PDE-constrained problems under uncertainty have not
	been analyzed, but serve as yet another class of  nonconvex
	problems and arise in many applications.
\end{enumerate}

The consistency of SAA solutions and SAA stationary points 
of nonlinear possibly
infinite-dimensional stochastic programs are typically 
based on some form of the feasible set's compactness.
In particular, 
the consistency results for SAA stationary points of infinite-dimensional
stochastic programs established in \cite{Balaji2008,Teran2010}
require the feasible set's compactness.
However, compact sets in infinite dimensions may regarded
as rare phenomena, given the fact that closed unit balls in Banach spaces
are compact if and only if the spaces are finite-dimensional.
This fact complicates the consistency analysis 
for nonconvex stochastic programs.
For strongly convex risk-neutral PDE-constrained optimization, 
the SAA approach has been analyzed in
\cite{Hoffhues2020,Martin2021,Roemisch2021,Milz2021,Milz2022c}
without requiring the compactness of the feasible set.

Contributions to analyzing the SAA approach for optimization problems
governed  by 
nonlinear operator equations with random inputs are relatively recent.
The SAA objective function's almost sure  
epiconvergence  and the SAA stationary points'  weak consistency
is analyzed in \cite{Phelps2016} for optimization problems governed by 
ordinary differential equations with random inputs. 
Utilizing epiconvergence, performance guarantees of optimal values 
of PDE-constrained optimization problems
under uncertainty are provided in \cite{Chen2021}
with respect to sample average approximations 
and control and state space discretizations, for example.
The asymptotic consistency of SAA optimal values and SAA solutions
to risk-averse PDE-constrained optimization problems has recently been
demonstrated in \cite{Milz2022a} using (Mosco-)epiconvergence,
and sample size estimates for risk-neutral semilinear PDE-constrained
problems are derived in \cite{Milz2022b}.

We use the compact operator $\coperator$ in
\eqref{eq:intro:subdifferentialidentity} to mathematically model problem
structure typically found  in  PDE-constrained optimization problems.
Related problem characteristics are 
used for establishing error estimates for finite-dimensional
approximations to PDE-constrained optimization problems
\cite{Casas2012a,Falk1973,Kroener2009} and
for algorithmic design and analysis in function space
\cite{Hertlein2019,Hintermueller2004,Hintermueller2016,Ulbrich2011},
for example.

The analysis of bilinear PDE-constrained optimization problems
is generally more 
challenging than those of their linear counterparts, as bilinear PDEs may lack
solutions for each point in the control space and bilinear
control problems are  generally nonconvex.
Bilinear control problems can be governed by
steady-state elliptic PDEs \cite{Kroener2009,Vallejos2008}, 
convection-diffusion equations \cite{Borzi2016},
and advection-reaction-diffusion
equations \cite{Glowinski2022}, for instance.
A substantial body of literature is available on the analysis,
algorithmic design, and applications of
deterministic bilinear control problems.
We refer the reader to
\cite{Casas2018a,Glowinski2022,Guillen-Gonzalez2020a}
for the analysis of deterministic bilinear control problems.
Finite element error estimates are established in
\cite{Casas2018a,Fuica2022,Kroener2009,Winkler2020,Yang2008} 
and algorithms are designed in
\cite{Borzi2016,Glowinski2022,Kahlbacher2012,Keung2000,%
	Kunisch1997,Vallejos2008}.
Bilinear control problems arise in medical image analysis \cite{Mang2018},
damping design \cite{Stojanovic1991}, 
groundwater remediation system design \cite{Guan1999},
optical flow \cite{Jarde2019}, 
PDE-constrained regression problems \cite{Nickl2020}, and 
chemorepulsion production \cite{Guillen-Gonzalez2020a}, for example.

\section*{Outline}
We introduce notation and terminology in \cref{sec:notation}.
\Cref{sec:consistentgeneral} establishes consistency of 
C-stationary points of optimization
problems obtained through approximations of infinite-dimensional nonconvex
optimization posed in reflexive Banach spaces. This basic result
is used to establish consistency of SAA C-stationary points
in \cref{sec:mcsaa} and of C-stationary points of approximate
stochastic programs defined
by weakly convergent probability measures in 
\cref{sec:wcsaa}. 
In  \cref{sect:avar}, we apply the framework developed in
\cref{sec:mcsaa} to risk-averse semilinear PDE-constrained
optimization using the average value-at-risk.
\Cref{subsec:blocp} discusses the application of our theory
to a risk-neutral bilinear PDE-constrained optimization problem.
We summarize our contributions and 
discuss open research questions  in \cref{sec:conclusion}.

\section{Notation and preliminaries}
\label{sec:notation}
Metric spaces are defined over the real numbers
and equipped with their Borel sigma-algebra
if not specified otherwise. 
For a Hilbert space $\hsp$, 
we denote by $\inner[\hsp]{\cdot}{\cdot}$ its inner product.
Let $X_1$ and $X_2$ be  Banach spaces.
The space of linear, bounded operators
from $X_1$ to $X_2$ is denoted by $\spL{X_1}{X_2}$.
An operator $A \in \spL{X_1}{X_2}$ is compact if
the image $A(W)$ is precompact in $X_2$
for each bounded set $W \subset X_1$.
The operator 
$A^* \in \spL{X_2^*}{X_1^*}$ is the adjoint operator
of $A \in \spL{X_1}{X_2}$.
We denote by $\Du \obj$ the \frechet\ derivative
of $\obj$ and by $\Du_y \obj $ or $\obj_y$ the \frechet\ derivative
with respect to $y$.
The dual  to a Banach space $\bsp$ is 
$\bsp^*$ and we use $\dualp[\bsp]{\cdot}{\cdot}$ 
to denote 
the dual pairing  between $\bsp^*$ and $\bsp$.
Let $X_0 \subset X$ be an open, nonempty subset
of a Banach space $X$. 
We denote Clarke's generalized directional derivative
of a locally Lipschitz continuous mapping $f \colon X_0 \to \real$
at $x \in X_0$ in the direction $h \in X$
by $f^\circ(x;h)$ 
and Clarke's subdifferential at $x$ by $\partial_C f(x) \subset X^*$
\cite[pp.\ 25 and 27]{Clarke1990}. We refer the reader to 
\cite[p.\ 39]{Clarke1990} for the definition of Carke regularity
of a function $f$. 
We use $\partial f$ to denote the usual convex subdifferential of
a function $f$. We define
the distance $\dist{v}{\Upsilon}$ from $v \in \Lambda
\subset X$ to $\Upsilon \subset X$ by
\begin{align*}
	\dist{v}{\Upsilon} = \inf_{w\in \Upsilon}\, \norm[X]{v-w}
	\quad \text{if} \quad   \Upsilon \neq \emptyset, 
	\quad \text{and} \quad  \dist{v}{\Upsilon} = \infty
	\quad \text{otherwise}. 
\end{align*} 
and the deviation $\deviation{\Lambda}{\Upsilon}$ between the sets 
$\Lambda$ and $\Upsilon$ by
$\deviation{\Lambda}{\Upsilon} = \sup_{v\in \Lambda}\, 	\dist{v}{\Upsilon}$.
Here $\norm[X]{\cdot}$ is a norm of $X$.
If $\phi \colon X \to (-\infty, \infty]$, then  
$\dom{\phi} = \{x \in X \colon \phi(x) < \infty\}$
denotes its domain.
Following \cite[p.\ 147]{Borwein1994}
and \cite[p.\ 306]{Borwein2010}, we say that a proper function
$\phi \colon X \to (-\infty, \infty]$
has the Kadec property if for each
$\bar x \in \dom{\phi}$, we have $x_k \to \bar x$
whenever $\phi(x_k) \to \phi(\bar x)$, $(x_k) \subset \dom{\psi}$, 
and $x_k \wto  \bar x$. Here $\to$ denotes strong convergence
and $\wto$ denotes weak convergence. 
For $x \in X$ and $\varepsilon \geq 0$, 
we denote by $\mathbb{B}_{X}(x;\varepsilon)$ the open ball about 
$x$ in $X$ with radius $\varepsilon$
and by $\overline{\mathbb{B}}_{X}(x;\varepsilon)$ its closure in $X$.

The following notation is used in \cref{subsec:blocp,sect:avar}.
Let $\domain \subset \real^d$ be a bounded domain.
The space $C^{0,1}(\bar{\domain})$ is the set of Lipschitz continuous
functions on $\bar{\domain}$, $L^p(\domain)$ ($1 \leq p \leq \infty$)
and $H^1(\domain)$, $H^2(\domain)$, and $H_0^1(\domain)$
are the ``usual'' Lebesgue and Sobolev spaces, respectively.
The space $H_0^1(\domain)$ is equipped with the norm
$\norm[H_0^1(\domain)]{v} = \norm[L^2(\domain)^d]{\nabla v}$
and the dual to $H_0^1(\domain)$ is denoted by $H^{-1}(\domain)$.
Moreover, let $C_\domain \in (0,\infty)$ be
\friedrichs ' constant of
$\domain$.
If $X$ is a  reflexive Banach space, then 
we identify  $(X^*)^*$ with $X$ and write $(X^*)^* = X$.
Moreover, we identify $L^2(\domain)^*$ with $L^2(\domain)$
and write $L^2(\domain)^*=L^2(\domain)$.
We denote by $\iota_0 \colon H_0^1(\domain) \to L^2(\domain)$
the embedding operator of the compact embedding
$H_0^1(\domain) \embedding L^2(\domain)$
and by $\iota_1 \colon H^1(\domain) \to L^2(\domain)$
that of 
$H^1(\domain) \embedding L^2(\domain)$.
We introduce further application-specific notation in \cref{subsec:blocp}.

\section{Consistency of C-stationary points}
\label{sec:consistentgeneral}

We consider the optimization problem 
\begin{align}
	\label{eq:probf}
	\min_{x \in \dom{\psi}}\, f(x) + \psi(x)
\end{align}
and its approximations
\begin{align}
	\label{eq:probfk}
	\min_{x \in \dom{\psi}}\, f_k(x) + \psi(x).
\end{align}
We explicitly add $\dom{\psi}$ as a constraint sets
in \eqref{eq:probf} and \eqref{eq:probfk}, as 
$f$ and $f_k$ may only be finite-valued on
an open neighborhood of $\dom{\psi}$.
We provide conditions sufficient for the asymptotic consistency
of C-stationary points of  the approximated problems \eqref{eq:probfk} to 
those of ``true'' problem \eqref{eq:probf}.  Our analysis is inspired by those
in \cite{Shapiro2003,Shapiro2007,Chen2012a,Milz2022a} 
and based on technical assumptions used in \cite{Milz2022}.
Using these consistency results, 
we establish asymptotic consistency of  C-stationary
points  of infinite-dimensional stochastic programs in 
\cref{sec:mcsaa,sec:wcsaa}.

\begin{assumption}
	\label{assumption:consistency}
	\begin{enumthm}[leftmargin=*,nosep]
		\item 
		\label{assumption:consistency:bspaces}
		The spaces $X$ and $V$ are   Banach spaces, 
		$X$ is reflexive, and $X_0 \subset X$ is nonempty and open.
		\item 
		\label{assumption:consistency:kadec}
		The function $\psi \colon X \to (-\infty, \infty]$ is proper, 
		convex, and lower semicontinuous. Moreover, $\psi$ 
		has the Kadec property  and $\dom{\psi} \subset X_0$.
		\item 
		\label{assumption:consistency:lipschitz}
		The functions $f \colon X_0 \to \real$ 
		and  $f _k \colon X_0 \to \real$, 
		$k \in \mathbb{N}$, are locally Lipschitz continuous.
		\item 
		\label{assumption:consistency:subdifferentialcompactness}
		The operator $\coperator  \in \spL{V}{X^*}$ is compact, 
		$M_k \colon \dom{\psi} \rightrightarrows V$
		is a set-valued mapping,  and $\partial_C f_k(x)
		= \coperator M_k(x)$ for all $x \in \dom{\psi}$ and 
		$k \in \mathbb{N}$.
		\item 
		\label{assumption:consistency:subdifferentialboundedness}
		For each $\bar x \in \dom{\psi}$, 
		each sequence $(x_k) \subset \dom{\psi}$
		with $ x_k \wto \bar x$, 
		and each sequence $(v_k)$ with $v_k \in M_k(x_k)$
		for all $k \in \mathbb{N}$, 
		$
		\limsup_{k \to \infty} \, \norm[V]{v_k}
		< \infty
		$.
		\item 
		\label{assumption:consistency:generalizedderivativeconsistency}
		For each $h \in X$, 
		each $\bar x \in \dom{\psi}$,  and each sequence $(x_k) \subset \dom{\psi}$
		with $x_k \to \bar x$, 
		it holds that
		$\limsup_{k \to \infty} f_k^\circ(x_k; h) \leq f^\circ(\bar x;h)$.
	\end{enumthm}
\end{assumption}

Under 
\Cref{assumption:consistency:bspaces,assumption:consistency:lipschitz,%
	assumption:consistency:kadec},
Clarke's subdifferential $\partial_C f_k(x)$
for each $x \in X_0$
is nonempty, convex, bounded, and weakly$^*$-compact
\cite[Prop.\ 2.1.2]{Clarke1990}.
Hence the set-valued mapping  $M_k$ in 
\Cref{assumption:consistency:subdifferentialcompactness} is
nonempty-valued.
If  $M_k$ has bounded images, then
\Cref{assumption:consistency:subdifferentialcompactness} 
ensures that $\partial_C f_k(x)$ is precompact-valued.
For Monte Carlo sample-based approximations of stochastic
programs, we verify
\Cref{assumption:consistency:generalizedderivativeconsistency} 
using an epigraphical law of large numbers \cite{Artstein1995}, 
basic properties of  Clarke's generalized directional derivative, 
and by imposing Clarke regularity of integrands
studied in \cref{sec:mcsaa}.

\Cref{assumption:consistency:subdifferentialcompactness} is satisfied 
for large problem classes as demonstrated next.

\begin{example}
	Let \Cref{assumption:consistency:bspaces,assumption:consistency:kadec,%
		assumption:consistency:lipschitz} hold.
	If $W$ is a  Banach space, 
	$\iota  \in \spL{X}{W}$ 
	is compact, $F_k \colon W \to \real$ is locally Lipschitz continuous
	and Clarke regular, 
	and $f_k \colon X \to \real$ is defined by $f_k(x) = F_k(\iota x)$, 
	then $\partial_C f_k(x) = \iota^* \partial_C F_k(\iota x)$
	\cite[Thm.\ 2.3.10]{Clarke1990}, and we may choose $B = \iota^*$
	and $M_k(x)  = \partial_C F_k(\iota x)$ 
	to satisfy \Cref{assumption:consistency:subdifferentialcompactness}.
	Composite objective functions
	with compact linear operators
	arise in nonsmooth optimal control
	\cite{Hertlein2022,Hertlein2019,Hintermueller2016,Surowiec2018}, 
	for example.
\end{example}

Next we discuss \Cref{assumption:consistency:subdifferentialboundedness,%
	assumption:consistency:subdifferentialcompactness}
when $X = \real^n$.

\begin{remark}
	Let \Cref{assumption:consistency:bspaces,assumption:consistency:kadec,%
		assumption:consistency:lipschitz} hold.
	Let  $X = \real^n$ 
	(equipped with the Euclidean norm $\norm[2]{\cdot}$)
	and let $X^* = V$. We identify  the dual to $\real^n$ with $\real^n$. Let
	$B \colon \real^n \to \real^n$ be the identity mapping.
	Then \Cref{assumption:consistency:subdifferentialcompactness} holds true
	with $ M_k = \partial_C f_k$.
	Next we discuss \Cref{assumption:consistency:subdifferentialboundedness}.
	Let $\bar x \in \dom{\psi}$ and let $L_k \geq 0$
	be Lipschitz constant of $f_k$
	on  $\mathbb{B}_{X}(\bar x;\varepsilon)$,
	where $\varepsilon = \varepsilon(\bar x)> 0$.
	If $\limsup_{k \to \infty}\, L_k < \infty$, 
	then \Cref{assumption:consistency:subdifferentialboundedness} holds true.
	Let us verify this assertion. 
	Let $(x_k) \subset \dom{\psi}$
	with $ x_k \wto \bar x$, 
	and let $(v_k)$ with $v_k \in M_k(x_k)$
	for all $k \in \mathbb{N}$. 
	Then $x_k \to \bar x$ and
	there exists $K = K(\bar x)$
	such that $x_k \in \mathbb{B}_{X}(\bar x;\varepsilon)$ for all $k \geq K$.
	Since $\partial_C f_k(x_k)  = M_k(x_k)$, 
	\cite[Prop.\ 2.1.2]{Clarke1990} ensures
	$\norm[2]{v_k} \leq L_k$ for all $k \geq K$.
	Hence 
	$
	\limsup_{k \to \infty} \, \norm[2]{v_k}
	< \infty
	$.
\end{remark}

\Cref{thm:genericconsistency} establishes consistency
of approximate C-stationary points. 
Let  $\mathcal{C}$ be  the set of 
C-stationary points of \eqref{eq:probf}, 
that is, $\mathcal{C}$ is
the set of all points
$x \in \dom{\psi}$ with $0 \in \partial_C f(x) + \partial \psi(x)$.
Let $\varepsilon \geq 0$ and let 
$\mathcal{C}_k^{\varepsilon}$
be  the set of all points $x \in \dom{\psi}$ with 
$0 \in  \bigcup_{ y \in \overline{\mathbb{B}}_{X}(x;\varepsilon) 
	\cap\dom{\psi}} \partial_C f_k(y) + \partial \psi(x)$
(cf.\ \cite[eqns.\ (3.2) and (4.2)]{Shapiro2007}).
Each point in $\mathcal{C}_k^{\varepsilon}$ is referred to as
an approximate C-stationary point of \eqref{eq:probfk}.

\begin{theorem}
	\label{thm:genericconsistency}
	Let \Cref{assumption:consistency} hold
	and let $(\varepsilon_k) \subset [0,\infty)$ satisfy $\varepsilon_k \to 0$
	as $k \to \infty$.	
	Moreover, let $\mathcal{C}_k^{\varepsilon_k}$ 
	be nonempty for each $k \in \mathbb{N}$. 
	\begin{enumerate}
		\item[\textnormal{(i)}]  If $(x_k)$ is a bounded sequence with
		$x_k \in \mathcal{C}_k^{\varepsilon_k}$  
		for each $k \in \mathbb{N}$, then $\mathcal{C}$ is nonempty
		and $\dist{x_k}{\mathcal{C}} \to 0$ as $k \to \infty$.
		\item[\textnormal{(ii)}] If there exists a bounded
		set $\mathscr{C} \subset X$
		with $\mathcal{C}_k^{\varepsilon_k}\subset \mathscr{C}$
		for all sufficiently large $k \in \mathbb{N}$, then
		$\deviation{\mathcal{C}_k^{\varepsilon_k}}{\mathcal{C}} \to 0$
		as $k \to \infty$.
	\end{enumerate}
\end{theorem}

\begin{proof}
	\Cref{assumption:consistency:lipschitz} ensures
	that the Clarke subdifferentials of $f$
	and $f_k$ are well-defined \cite[Prop.\ 2.1.2]{Clarke1990}. 
	
	\textnormal{(i)} 
	Let $(\dist{x_k}{\mathcal{C}})_{K_0}$ be a subsequence of
	$(\dist{x_k}{\mathcal{C}})  \subset [0,\infty]$.
	Since $x_k \in \mathcal{\mathcal{C}}_k^{\varepsilon_k}$,
	there exists $y_k \in \dom{\psi}$
	with $\norm[X]{x_k-y_k} \leq \varepsilon_k$
	and $g_k \in  \partial_C f_k(y_k)$ with
	\begin{align}
		\label{eq:Jun262023630}
		\dualp[X]{g_k}{x-x_k} + \psi(x) \geq \psi(x_k)
		\quad \text{for all} \quad x \in X.
	\end{align}
	The boundedness of  $(x_k)$ implies that of 
	$(x_k)_{K_0}$. Combined with the fact that $X$ is reflexive
	(see \Cref{assumption:consistency:bspaces})
	and that $\dom{\psi}$ is weakly sequentially closed
	(see \Cref{assumption:consistency:kadec} and
	\cite[Thm.\ 2.23 (ii)]{Bonnans2013}), we find that
	$(x_k)_{K_0}$ has a subsequence 
	$(x_{k})_{K_1}$ with $x_{k} \wto \bar x \in \dom{\psi}$
	as $K_1 \ni k \to \infty$.
	Since $\norm[X]{x_k-y_k} \leq \varepsilon_k$
	and $\varepsilon_k \to 0$,
	we find that $y_{k} \wto \bar x$
	as $K_1 \ni k \to \infty$.
	Since $g_k \in \partial_C f_k(x_k)$, 
	\Cref{assumption:consistency:subdifferentialcompactness} 
	ensures $g_k = \coperator v_k$ for some $v_k \in M_k(y_k)$.
	\Cref{assumption:consistency:subdifferentialboundedness} implies that
	$(v_K)_{K_1}$ is bounded.
	Combined with the compactness of $\coperator$
	(see \Cref{assumption:consistency:subdifferentialcompactness}), 
	we obtain that $(g_k)_{K_1} \subset X^*$ has a further subsequence
	$(g_k)_{K_2}$ with $g_k \to \bar g$ as $K_2 \ni k \to \infty$
	\cite[Thm.\ 8.1-5]{Kreyszig1978}.
	Hence 
	$\dualp[X]{g_k}{x-x_k} \to \dualp[X]{\bar g}{x-\bar x}$
	as $K_2 \ni k \to \infty$
	\cite[Thm.\ 2.23 (iv)]{Bonnans2013}.
	\Cref{assumption:consistency:kadec} implies that $\psi$
	is weakly lower semicontinuous. Hence
	$\liminf_{K_2 \ni k \to \infty}\, \psi(x_k) \geq \psi(\bar x)$.
	Putting together the pieces and
	using \eqref{eq:Jun262023630}, we have
	\begin{align}
		\label{eq:varineqbarx}
		\dualp[X]{\bar g}{x-\bar x} + \psi(x) \geq \psi(\bar x)
		\quad \text{for all} \quad x \in X.
	\end{align}
	Using \eqref{eq:Jun262023630} once more, we obtain
	$\psi(\bar x) = \lim_{K_2 \ni k \to \infty}\dualp[X]{g_k}{\bar x-x_k} + \psi(\bar x) \geq 
	\limsup_{K_2 \ni k \to \infty}\psi(x_k)$. 
	Hence $\psi(x_k) \to \psi(\bar x)$ as $K_2 \ni k \to \infty$.
	Since $\psi$ is Kadec according to
	\Cref{assumption:consistency:kadec}, $x_k \to \bar x$ 
	as $K_2 \ni k \to \infty$.
	Since $\norm[X]{x_k-y_k} \leq \varepsilon_k$, 
	we also have $y_{k} \to \bar x$
	as $K_2 \ni k \to \infty$.
	
	We define $\widetilde{x}_k = y_k$ if $k \in K_2$
	and $\widetilde{x}_k = \bar x$ otherwise.
	We have $\widetilde{x}_k \to \bar x$ as $k \to \infty$.
	Fix $h \in X$.
	Since $g_k \in  \partial_C f_k(y_k)$, 
	$f_k^\circ(y_k; h) \geq \dualp[X]{g_k}{h}$
	\cite[Prop.\ 2.1.5]{Clarke1990}. 
	Combined 
	with \Cref{assumption:consistency:generalizedderivativeconsistency}, 
	we have
	\begin{align*}
		\dualp[X]{\bar g}{h} = \lim_{K_2 \ni k \to \infty}\, 
		\dualp[X]{g_k}{h}
		\leq 
		\limsup_{K_2 \ni k \to \infty}\,f_k^\circ(y_k; h)
		\leq 
		\limsup_{k \to \infty}\,f_k^\circ(\widetilde{x}_k; h) 
		\leq f^\circ(\bar x; h).
	\end{align*}
	Hence $\bar g \in \partial_C f(\bar x)$
	\cite[Prop.\ 2.1.5]{Clarke1990}.
	Now \eqref{eq:varineqbarx} ensures
	$0 \in \partial_C f(\bar x) + \partial \psi(\bar x)$,
	that is, $\bar x \in \mathcal{C}$. Hence
	$\dist{x_k}{\mathcal{C}} \leq \norm[X]{x_k-\bar x} \to 0$
	as $K_2 \ni k \to 0$,
	and $(\dist{x_k}{\mathcal{C}})$ is bounded.
	Our derivations also show that each subsequence
	of $(\dist{x_k}{\mathcal{C}})$
	has a further subsequence converging to zero.
	Hence  $\dist{x_k}{\mathcal{C}} \to 0$ as $k \to \infty$.
	
	\textnormal{(ii)}
	Since 
	$\mathcal{C}_k^{\varepsilon_k} \subset \mathscr{C}$, and
	$\mathcal{C}_k^{\varepsilon_k}$ and $\mathcal{C}$
	are nonempty, 
	and $\mathscr{C}$ is bounded, we have
	$\deviation{\mathcal{C}_k^{\varepsilon_k}}{\mathcal{C}}
	\leq \deviation{\mathscr{C}}{\mathcal{C}} < \infty$
	for all sufficiently large $k \in \mathbb{N}$. 
	Hence for all sufficiently large $k \in \mathbb{N}$, 
	there exists
	$x_k \in \mathcal{C}_k^{\varepsilon_k}$
	with 
	$\deviation{\mathcal{C}_k^{\varepsilon_k}}{\mathcal{C}}
	\leq \dist{x_k}{\mathcal{C}} + 1/k$.
	Using part~\textnormal{(i)}, we
	have $\dist{x_k}{\mathcal{C}} \to 0$ as $k \to \infty$.
	Hence 	
	$\deviation{\mathcal{C}_k^{\varepsilon_k}}{\mathcal{C}} \to 0$
	as $k \to \infty$.
\end{proof}

\section{Empirical approximations via Monte Carlo sampling}
\label{sec:mcsaa}
Let $\Xi$ be a complete, separable metric space,
and let $(\Xi, \mathcal{A}, \mathbb{P})$ be a complete
probability space. 
Let $\xi^1$, $\xi^2, \ldots$ 
be independent identically distributed
$\Xi$-valued random elements defined on a complete
probability space $(\Omega, \mathcal{F}, P)$ such that each $\xi^i$ has
distribution $\mathbb{P}$.

We consider the risk-neutral optimization problem
\begin{align}
	\label{eq:probmo}
	\min_{x \in \dom{\psi}}\,  \int_{\Xi}\, \pobj_{\xi}(x)\, \du \mathbb{P}(\xi)
	+ \psi(x)
\end{align}
and its SAA problem 
\begin{align}
	\label{eq:probaamo}
	\min_{x \in \dom{\psi}}\,  \frac{1}{N} \sum_{i=1}^N \pobj_{\xi^i}(x)
	+ \psi(x).
\end{align}

We study the almost sure convergence of SAA C-stationary points of
\eqref{eq:probaamo} to the C-stationary point's set of \eqref{eq:probmo}.
We analyze the consistency under conditions related to those in 
\Cref{assumption:consistency}.

\begin{assumption}
	\label{assumption:mcsampling}
	\begin{enumthm}[leftmargin=*,nosep]
		\item 
		\label{assumption:mcsampling:spaces}
		The spaces $X$ and $V$ are  separable Banach spaces, 
		$X$ is reflexive, and $X_0 \subset X$ is nonempty and open.
		\item 
		\label{assumption:mcsampling:psi}
		The function $\psi \colon X \to (-\infty, \infty]$ is proper, 
		convex, and lower semicontinuous. Moreover, $\psi$ 
		has the Kadec property  and $\dom{\psi} \subset X_0$.

		\item 
		\label{assumption:mcsampling:integrable}
		For each $\bar x \in \dom{\psi}$, 
		there exists an open neighborhood $\mathcal{V}_{\bar x} \subset X_0$
		of $\bar x$ and a $\mathbb{P}$-integrable 
		random variable $L_{\bar x} \colon \Xi \to [0,\infty)$
		such that for each $\xi \in \Xi$, $\pobj_{\xi}\colon X_0 \to \real$ 
		is Lipschitz continuous on
		$\mathcal{V}_{\bar x}$ with Lipschitz constant 
		$L_{\bar x}(\xi)$. 
		
		\item 
		\label{assumption:mcsampling:regular}
		For each $\xi \in \Xi$, 
		$\pobj_{\xi} \colon X_0 \to \real$ is  Clarke regular. 
		
		\item 
		\label{assumption:mcsampling:mi}
		For each $x \in X_0$, 
		$\Xi \ni \xi \mapsto \pobj_{\xi}(x)\in \real$ is measurable
		and $\mathbb{P}$-integrable.
		
		\item 
		\label{assumption:mcsampling:subdifferentialidentify}
		
		The operator $\coperator \in \spL{V}{X^*}$ is compact, 
		$M_\xi \colon \dom{\psi}  \rightrightarrows V$ is a 
		set-valued map with nonempty
		images for each $\xi \in \Xi$, and 
		$\partial_C \pobj_{\xi}(x)= \coperator M_\xi(x)$
		for each $(x,\xi) \in \dom{\psi} \times \Xi$.
		\item 
		\label{assumption:mcsampling:subdifferentialbounded}
		There exists  $\Omega_0  \subset \Omega$
		with $\Omega_0 \in \mathcal{F}$ and $P(\Omega_0) = 1$
		such that  for each $\bar x \in \dom{\psi}$, 
		each sequence $(x_N) \subset \dom{\psi}$
		with $x_N \wto \bar x$, 
		and each sequence $(v_N) \subset V$,
		we have
		$
		\limsup_{N \to \infty} \, \norm[V]{v_N}
		< \infty
		$
		whenever $\omega \in \Omega_0$ and
		$v_N \in (1/N)\sum_{i=1}^N
		M_{\xi^i(\omega)}(x_N)$
		for each $N \in \mathbb{N}$.
	\end{enumthm}
\end{assumption}

\Cref{assumption:mcsampling:spaces,assumption:mcsampling:psi}
correspond to 
\Cref{assumption:consistency:bspaces,%
	assumption:consistency:kadec}, but we  impose in 
addition separability of $X$ and $V$.
\Cref{assumption:mcsampling:spaces,%
	assumption:mcsampling:regular,assumption:mcsampling:integrable}
allow us to use Clarke's calculus for generalized gradients of integral
functions  \cite[sect.\ 2.7]{Clarke1990}.
To apply the epigraphical law of large numbers derived in 
\cite{Artstein1995}, we use
\Cref{assumption:mcsampling:regular,assumption:mcsampling:integrable}.
\Cref{assumption:mcsampling:subdifferentialidentify,%
	assumption:mcsampling:subdifferentialbounded} generalize
technical conditions used in \cite{Milz2022}.

\begin{lemma}
	\label{lem:mcsampling:subdifferentialbounded}
	Let $\zeta \colon \Xi \to [0,\infty)$ be measurable
	and $\mathbb{P}$-integrable, 
	and let \Cref{assumption:mcsampling:spaces,%
		assumption:mcsampling:subdifferentialidentify,%
		assumption:mcsampling:psi} hold.
	If  $\norm[V]{m_\xi(x)} \leq \zeta(\xi)$
	for all $m_\xi(x)\in M_\xi(x)$, each $x \in \dom{\psi}$, 
	and every $\xi \in \Xi$, then
	\Cref{assumption:mcsampling:subdifferentialbounded} 
	holds true.
\end{lemma}
\begin{proof}
	Since $\zeta(\xi^1), \zeta(\xi^2), \ldots$ are independent,
	the law of large numbers ensures that
	$(1/N) \sum_{i=1}^N \zeta(\xi^i)$
	converges almost surely to 
	$\int_{\Xi} \zeta(\xi) \, \du \mathbb{P}(\xi)$
	as $N \to \infty$. Hence there exists
	a set $\Omega_0  \subset \Omega$
	with $\Omega_0 \in \mathcal{F}$, $P(\Omega_0) = 1$,
	and for each $\omega \in \Omega_0$, 
	we have $(1/N) \sum_{i=1}^N \zeta(\xi^i(\omega))
	\to \int_{\Xi} \zeta(\xi) \, \du \mathbb{P}(\xi)$
	as $N \to \infty$.
	
	Fix $\omega \in \Omega_0$.
	Let  $\bar x \in \dom{\psi}$, 
	let  $(x_N) \subset \dom{\psi}$
	fulfill $x _N \wto \bar x$, 
	and let $(v_N)$ satisfy
	$v_N \in (1/N)\sum_{i=1}^N
	M_{\xi^i(\omega)}(x_N)$
	for all $N \in \mathbb{N}$. 
	The latter ensures the
	existence of $v_N^i \in M_{\xi^i(\omega)}(x_N)$
	with $v_N = (1/N)\sum_{i=1}^N v_N^i$. 
	Now the lemma's hypotheses ensure 
	$\norm[V]{v_N^i} \leq \zeta(\xi^i)$. Hence
	$
	\norm[V]{v_N}
	\leq (1/N) \sum_{i=1}^N \norm[V]{v_N^i}
	\leq (1/N)\sum_{i=1}^N \zeta(\xi^i)
	$.
	We obtain $
	\limsup_{N \to \infty} \, \norm[V]{v_N}
	< \infty
	$.
\end{proof}

We impose a technical measurability condition on 
Clarke's generalized directional derivative of
$\pobj_{\xi}$  in \Cref{thm:consistencymc}. 
The measurability of Clarke's subdifferential 
as a function of the decision variables 
and parameters has been analyzed, for example, in
\cite[Lem.\ 4]{Norkin1986} and \cite[Prop.\ 2.1]{Xu2009}.

Let  $\mathcal{C}$ be  the set of all C-stationary points
of \eqref{eq:probmo}, that is, the set of all 
$x \in \dom{\psi}$ with 
$0 \in \partial_C \big[\int_{\Xi}\, \pobj_{\xi}(x)\, \du \mathbb{P}(\xi)\big] + \partial \psi(x)$.
Let $\varepsilon \geq 0$ and for each $\omega \in \Omega$, let 
$\hat{\mathcal{C}}_N^{\varepsilon}(\omega)$
be  the set of all points $x \in \dom{\psi}$ with 
$$0 \in  \bigcup_{ v \in \overline{\mathbb{B}}_{X}(x;\varepsilon) \cap\dom{\psi}}
\partial_C  \Big[\frac{1}{N} \sum_{i=1}^N \pobj_{\xi^i(\omega)}(v)\Big] 
+ \partial \psi(x);$$
cf.\ \cite[eqns.\ (3.2) and (4.2)]{Shapiro2007}.
Each element in $\hat{\mathcal{C}}_N^{\varepsilon}(\omega)$
is referred to as an approximate C-stationary point of 
\eqref{eq:probaamo}. 

\begin{theorem}
	\label{thm:consistencymc}
	Let \Cref{assumption:mcsampling} hold true
	and  let $(\varepsilon_N) \subset [0,\infty)$ 
	be a sequence with $\varepsilon_N \to 0$
	as $N \to \infty$. Suppose that 
	$\dom{\psi} \times X \times \Xi \ni 
	(x,h,\xi) \mapsto \pobj_{\xi}^\circ(x;h)$ 
	is measurable, and let $\mathcal{C}$
	and $\hat{\mathcal{C}}_N^{\varepsilon_N}(\omega)$ 
	be nonempty for each $N \in \mathbb{N}$ and $\omega \in \Omega$. 
	\begin{enumerate}
		\item[\textnormal{(i)}] If 
		$\Omega_1 \subset \Omega$
		with $\Omega_1 \in \mathcal{F}$ and $P(\Omega_1) =1$ and
		for each $\omega \in \Omega_1$,
		$(x_N(\omega))$ is a bounded sequence with
		$x_N(\omega) \in \hat{\mathcal{C}}_N^{\varepsilon_N}(\omega)$  
		for all sufficiently large $N \in \mathbb{N}$,  then 
		with probability one (\wpone), 
		$\dist{x_N}{\mathcal{C}} \to 0$ as $N \to \infty$.
		\item[\textnormal{(ii)}]  If there exists a bounded
		set $\mathscr{C} \subset X$
		such that \wpone, $\hat{\mathcal{C}}_N^{\varepsilon_N}\subset \mathscr{C}$
		for all sufficiently large $N\in \mathbb{N}$, then \wpone, 
		$\deviation{\hat{\mathcal{C}}_N^{\varepsilon_N}}{\mathcal{C}} \to 0$
		as $N \to \infty$.
	\end{enumerate}
\end{theorem}

\begin{proof}
	We establish the assertions via applications
	of \Cref{thm:genericconsistency}. 
	
	First, we observe that for each $\omega \in \Omega_0$,
	\Cref{assumption:mcsampling:subdifferentialbounded} 
	implies 
	\Cref{assumption:consistency:subdifferentialboundedness}
	with $M_k(\cdot)$ replaced 
	by
	$\hat{M}_N(\cdot, \omega)= 
	(1/N)  \sum_{i=1}^N M_{\xi^i(\omega)}(\cdot)$.
	
	Second, 
	\Cref{assumption:mcsampling:regular,assumption:mcsampling:integrable,%
		assumption:mcsampling:subdifferentialidentify},
	and the sum rule
	\cite[Cor.\ 3 on p. 40]{Clarke1990} ensure
	for all $\omega \in \Omega$, $x \in X_0$, 
	and every $N \in \mathbb{N}$,
	\begin{align*}
		\partial_C  \Big[\frac{1}{N} \sum_{i=1}^N \pobj_{\xi^i(\omega)}(x)\Big] 
		= \frac{1}{N} \sum_{i=1}^N \partial_C  \pobj_{\xi^i(\omega)}(x)
		=  \frac{1}{N} \sum_{i=1}^N  \coperator M_{\xi^i(\omega)}(x)
		= \coperator \hat{M}_N(x, \omega).
	\end{align*}
	Hence for all $\omega \in \Omega$,
	\Cref{assumption:consistency:subdifferentialcompactness} holds true
	with $M_k(\cdot)$ replaced  by
	$\hat{M}_N(\cdot, \omega)$.
	
	Third,
	we show that for almost every $\omega \in \Omega$,
	\Cref{assumption:consistency:generalizedderivativeconsistency}
	is satisfied with $f(\cdot)$ replaced 
	by the expectation function
	$\int_{\Xi} \pobj_\xi(\cdot) \, \du \mathbb{P}(\xi)$
	and with $f_k(\cdot)$ replaced by the SAA objective function
	$(1/N) \sum_{i=1}^N \pobj_{\xi^i(\omega)}(\cdot)$.
	For each $\xi \in \Xi$, 
	$X_0 \times X \ni (x,h) \mapsto  \pobj_{\xi}^{\circ}(x; h)$ 
	is upper semicontinuous
	\cite[Prop.\ 2.1.1]{Clarke1990}. 
	Fix $\bar x \in \dom{\psi}$. 
	Let $\mathcal{V}_{\bar x}$ be an open neighborhood
	and let $L_{\bar x}$ be a Lipschitz constant given by
	\Cref{assumption:mcsampling:integrable}. 
	Using \cite[Prop.\ 2.1.1]{Clarke1990}, 
	we have
	$|\pobj_\xi^{\circ}(x; h)| \leq L_{\bar x}(\xi)$ 
	for each $(x, h, \xi) \in \mathcal{V}_{\bar x} \times X \times \Xi$.
	Moreover $(x,h,\xi) \mapsto \pobj_{\xi}^\circ(x;h)$ 
	is measurable by assumption.
	Combined
	with \cite[Thm.\ 2.3]{Artstein1995}, we find that \wpone, 
	$\dom{\psi} \times X \ni 
	(x,h) \mapsto (1/N)  \sum_{i=1}^N \pobj_{\xi^i}^{\circ}(x; h)$
	hypoconverges to 
	$(x,h) \mapsto \int_\Xi \pobj_\xi^{\circ}(x; h) \, \du  \mathbb{P}(\xi)$.\footnote{%
		We refer the reader to \cite[p.\ 3]{Artstein1995} for a definition
		of epiconvergence of a sequence of functions $(g_k)$ to some function
		$g$. A sequence of functions $(g_k)$ hypoconverges to $g$ if 
		(and only if)	$(-g_k)$ epiconverges to $-g$ 
		\cite[pp.\ 242--243]{Rockafellar2009}.} 
	Using \Cref{assumption:mcsampling:spaces}, 
	\cite[Cor.\ 3 on p.\ 40 and Thm.\ 2.7.2]{Clarke1990}  
	and Clarke regularity of $\pobj_{\xi}$
	for each $\xi \in \Xi$ (see \Cref{assumption:mcsampling:regular}), 
	we have the identities
	$\int_\Xi \pobj_\xi^{\circ}(x; h) \, \du  \mathbb{P}(\xi) =  
	(\int_{\Xi} \pobj_\xi(\cdot) \, \du \mathbb{P}(\xi))^{\circ}(x; h)
	$
	and $(1/N)  \sum_{i=1}^N \pobj_{\xi^i}^{\circ}(x; h)
	=
	\big(
	(1/N)  \sum_{i=1}^N \pobj_{\xi^i}(\cdot)
	\big )^{\circ}(x; h)
	$
	for every $(x,h) \in X_0 \times X$.
	Hence
	there exists $\Omega_2 \subset \Omega$
	with $\Omega_2 \in \mathcal{F}$ and $P(\Omega_2) = 1$
	such that for each $\omega \in \Omega_2$, 
	all $h \in X$ and each $(x_N)  \subset \dom{\psi}$
	with $x_N \to x$, we have
	$$\limsup_{N \to \infty}\, 
	\Big(\frac{1}{N}
	\sum_{i=1}^N \pobj_{\xi^i(\omega)}(\cdot)
	\Big )^{\circ}(x_N; h)
	\leq 
	\Big(\int_{\Xi} \pobj_\xi(\cdot) \, \du \mathbb{P}(\xi)\Big)^{\circ}(x; h).$$

	Now we establish parts (i) and (ii).
	
	\textnormal{(i)}
	By assumption and construction, we have
	$ \Omega_0 \cap \Omega_1 \cap \Omega_2 \in \mathcal{F}$.
	Moreover this event happens \wpone. 
	Combining our derivations, \Cref{thm:genericconsistency}
	ensures for each $ \omega \in \Omega_0 \cap \Omega_1 \cap \Omega_2$, 
	we have
	$\dist{x_N(\omega)}{\mathcal{C}} \to 0$ as $N \to \infty$.

	\textnormal{(ii)}
	Since \wpone, $\hat{\mathcal{C}}_N^{\varepsilon_N}\subset \mathscr{C}$
	for all sufficiently large $N\in \mathbb{N}$,
	there exists a set $\Omega_3 \in \mathcal{F}$ with
	$P(\Omega_3) = 1$ such that for each $\omega \in \Omega_3$,
	there exists $n(\omega) \in \mathbb{N}$
	with $\hat{\mathcal{C}}_N^{\varepsilon_N}(\omega)\subset \mathscr{C}$
	for all $N \geq n(\omega)$.
	In particular
	$\Omega_0 \cap \Omega_2 \cap \Omega_3 \in \mathcal{F}$
	and $P(\Omega_0 \cap \Omega_2 \cap \Omega_3 ) =1$.
	Now  \Cref{thm:genericconsistency} ensures
	for each $ \omega \in \Omega_0 \cap \Omega_2 \cap \Omega_3$, 
	$\deviation{\hat{\mathcal{C}}_N^{\varepsilon_N}(\omega)}{\mathcal{C}} \to 0$
	as $N \to \infty$.
\end{proof}

\section{Approximations via weakly converging probability measures}
\label{sec:wcsaa}
Motivated by recent contributions on epiconvergent discretizations of 
stochastic programs
\cite{Diem2022,Feinberg2022,Hess2019,Pennanen2005}, 
we establish consistency of approximations to 
the stochastic program \eqref{eq:probmo}
which result from approximating $\mathbb{P}$ by a sequence
of weakly convergent probability measures  $(\mathbb{P}_N)$.
Our main interest in considering weakly convergent
probability measures stems from those considered in
\cite[sect.\ 2]{Pennanen2005} and
\cite[sect.\ 3]{Koivu2005}. 
Throughout the section, 
we assume $\Xi$ be a complete, separable metric
space and $\mathbb{P}$ be a probability
measure on $\Xi$. 
We refer the reader to \cite[p.\ 65]{Kallenberg2002}
for the definition of weak convergence of probability measures. 

For a sequence $(\mathbb{P}_N)$ of probability measures
defined on $\Xi$, 
we approximate \eqref{eq:probmo} via the optimization problems
\begin{align}
	\min_{x \in \dom{\psi}}\,  \int_{\Xi}\, \pobj_{\xi}(x)\, 
	\du \mathbb{P}_N(\xi)
	+ \psi(x).
\end{align}

Our consistency analysis is built on the conditions in
\cref{assumption:wsampling} and those needed to apply
\cite[Cor.\ 3.4]{Langen1981}.

\begin{assumption}
	\label{assumption:wsampling}
	In addition to
	\Cref{assumption:mcsampling:spaces,assumption:mcsampling:psi,%
		assumption:mcsampling:subdifferentialidentify,assumption:mcsampling:mi}, 
	we consider the following conditions.
	\begin{enumthm}[leftmargin=*,nosep]
		\item 
		\label{assumption:wcsampling:weakconvergence}	
		For each $N \in \mathbb{N}$, 
		$\mathbb{P}_N$ is a probability measure on $\Xi$, 
		and 	
		$(\mathbb{P}_N)$ weakly converges
		to $\mathbb{P}$ as $N \to \infty$.
		\item 	
		\label{assumption:wcsampling:differentiable}
		For each $\xi \in \Xi$, 
		$\pobj_{\xi} \colon X_0 \to \real$ is  
		continuously
		differentiable.
		For every $x \in X_0$, 
		$\Xi \ni \xi \mapsto \pobj_{\xi}(x)\in \real$ 
		is $\mathbb{P}_N$-integrable.
		For all  $h \in X$,
		$X_0 \times \Xi \ni (x,\xi) \mapsto \dualp[X]{\Du\pobj_{\xi}(x)}{h} 
		\in \real$ 
		is upper semicontinuous.
		\item 	
		\label{assumption:wcsampling:integrable} 
		\Cref{assumption:mcsampling:integrable} holds true
		with  $L_{\bar x}$ being continuous,
		$\mathbb{P}_N$-integrable for each $N \in \mathbb{N}$, and satisfying
		$\limsup_{N \to \infty}\, 
		\int_{\Xi} L_{\bar x}(\xi) \, \du \mathbb{P}_N(\xi)
		\leq \int_{\Xi} L_{\bar x}(\xi) \, \du \mathbb{P}(\xi)
		$.
		\item For each $\bar x \in \dom{\psi}$, 
		each sequence $(x_N) \subset \dom{\psi}$
		with $x_N \wto \bar x$, 
		and each sequence $(v_N)$ with 
		$v_N \in \int_{\Xi} M_{\xi}(x_N) \, \du \mathbb{P}_N(\xi)$
		for all $N \in \mathbb{N}$, 
		$
		\limsup_{N \to \infty} \, \norm[V]{v_N}
		< \infty
		$.
	\end{enumthm}
\end{assumption}

While \Cref{assumption:wcsampling:differentiable} imposes 
continuity and differentiablity, these conditions
are satisfied for certain problem classes \cite{Milz2020a}.
However the continuity conditions may be relaxed in view of
\cite[Thms.\ 3.6 and 5.2]{Artstein1994}.
Under \Cref{assumption:wcsampling:differentiable}, 
the continuously 
differentiable function $\pobj_{\xi} \colon X_0 \to \real$ is Clarke regular
for each $\xi \in \Xi$ \cite[Prop.\ 2.3.6]{Clarke1990}
and $\partial_C \pobj_{\xi}(x)  = \{\, \Du \pobj_{\xi}(x) \, \}$
\cite[Cor.\ on p.\ 32 and Prop.\ 2.2.4]{Clarke1990}.
Using \Cref{assumption:wsampling} and \cite[Lem.\ C.3]{Geiersbach2020},
we find that
$X_0 \ni x \mapsto \int_{\Xi}\, \pobj_{\xi}(x)\, \du \mathbb{P}(\xi)$
and
$X_0 \ni x \mapsto \int_{\Xi}\, \pobj_{\xi}(x)\, \du \mathbb{P}_N(\xi)$
are \frechet\ differentiable
and we can interchange derivatives  and integrals.

Let  $\mathfrak{C}$ be  the set of all 
$x \in \dom{\psi}$ with 
$0 \in \Du \big[\int_{\Xi}\, \pobj_{\xi}(x)\, \du \mathbb{P}(\xi)\big]
+ \partial \psi(x)$.
Let $\varepsilon \geq 0$ and for each $\omega \in \Omega$, let 
$\mathfrak{C}_N^{\varepsilon}$
be  the set of all points $x \in \dom{\psi}$ with 
$0 \in  \bigcup_{ v \in \overline{\mathbb{B}}_{X}(x;\varepsilon) \cap\dom{\psi}}
\Du \big[\int_{\Xi}\, \pobj_{\xi}(v) \, \du \mathbb{P}_N(\xi)\big]
+ \partial \psi(x)$.

\begin{theorem}
	\label{thm:consistencyw}
	Let \Cref{assumption:wsampling} hold true
	and  let $(\varepsilon_N) \subset [0,\infty)$ 
	be a sequence with $\varepsilon_N \to 0$
	as $N \to \infty$. Let $\mathfrak{C}$
	and $\mathfrak{C}_N^{\varepsilon_N}$ 
	be nonempty for each $N \in \mathbb{N}$. 
	\begin{enumerate}
		\item[\textnormal{(i)} ] 	If $(x_N)$ is bounded  with
		$x_N \in \mathfrak{C}_N^{\varepsilon_N}$  
		for each $N \in \mathbb{N}$, then 
		$\dist{x_N}{\mathcal{C}} \to 0$ as $N \to \infty$.
		\item[\textnormal{(ii)}]If there exists a bounded
		set $\mathscr{C} \subset X$
		such that $\mathfrak{C}_N^{\varepsilon_N}\subset \mathscr{C}$
		for all sufficiently large $N\in \mathbb{N}$, then 
		$\deviation{\mathfrak{C}_N^{\varepsilon_N}}{\mathfrak{C}} \to 0$
		as $N \to \infty$.
	\end{enumerate}
\end{theorem}

\begin{proof}
	Using \cite[Cor.\ 3.4]{Langen1981}, we show that 
	\Cref{assumption:consistency:generalizedderivativeconsistency} 
	holds true with
	$f(\cdot)$ 
	replaced by $\int_{\Xi}\, \pobj_{\xi}(\cdot)\, \du \mathbb{P}(\xi)$, 
	and $f_k(\cdot)$
	replaced by
	$\int_{\Xi}\, \pobj_{\xi}(\cdot)\, \du \mathbb{P}_N(\xi)$.
	Fix $h \in X$, 
	let $\bar x \in \dom{\psi}$ and let $(x_N) \subset \dom{\psi}$
	with $x_N \to \bar x$ be arbitrary.
	According to \Cref{assumption:wcsampling:integrable},
	there exists $\varepsilon  = \varepsilon(\bar x)> 0$
	such $\pobj_\xi$
	is Lipschitz continuous on $\mathbb{B}_X(\bar x, \varepsilon)$
	with a Lipschitz constant $L_{\bar x}(\xi)$ for each $\xi \in \Xi$, 
	$L_{\bar x}$ is continuous, $\mathbb{P}$- and
	$\mathbb{P}_N$-integrable, and 
	$\limsup_{N \to \infty}\, 
	\int_{\Xi} L_{\bar x}(\xi) \, \du \mathbb{P}_N(\xi)
	\leq \int_{\Xi} L_{\bar x}(\xi) \, \du \mathbb{P}(\xi)
	$.
	The mean value theorem yields
	$|\dualp[X]{\Du\pobj_{\xi}(x)}{h}| \leq L_{\bar x}(\xi) \norm[X]{h}$
	for all $x \in \mathbb{B}_X(\bar x, \varepsilon)$, $h \in X$, and
	$\xi \in \Xi$.
	Combined with \Cref{assumption:wcsampling:weakconvergence,%
		assumption:wcsampling:differentiable}, 
	\cite[Cor.\ 3.4]{Langen1981} 
	yields
	\begin{align*}
		\limsup_{N \to \infty}\, 
		\int_{\Xi} \dualp[X]{\Du\pobj_{\xi}(x_N)}{h}\, \du \mathbb{P}_N(\xi)
		\leq \int_{\Xi} \dualp[X]{\Du\pobj_{\xi}(\bar x)}{h}\, 
		\du \mathbb{P}(\xi).
	\end{align*}
	Using the chain rule \cite[Lem.\ C.3]{Geiersbach2020}, 
	we can interchange derivatives and integrals in the above
	equation. Hence
	\Cref{assumption:consistency:generalizedderivativeconsistency} 
	holds true.
	
	Both assertions are now implied by \Cref{thm:genericconsistency}.
\end{proof}

\section{Application to risk-averse semilinear PDE-constrained
	optimization}
\label{sect:avar}
We show that our main result, \Cref{thm:consistencymc}, applies
to risk-averse PDE-constrained optimization using the average
value-at-risk.
For $\beta \in (0,1)$, and a $\mathbb{P}$-integrable random
variable $Z \colon \Xi \to \real$, 
the average value-at-risk
$\AVaR{}$  evaluated at $Z$ is defined by
\begin{align}
	\label{eq:avar}
	\AVaR{Z} = \inf_{t \in \real}\,
	\big\{\,t + \tfrac{1}{1-\beta} \mathbb{E}_{\mathbb{P}}[\maxo{Z-t}]\,
	\big\},
\end{align}
where $\maxo{s} =  \max\{0,s\}$ for $s \in \real$
\cite[eq.\ (6.23)]{Shapiro2021}.
We consider the risk-averse semilinear PDE-constrained problem
\begin{align}
	\label{eq:ocpspde}
	\min_{u \in \adcsp}\,
	(1/2) \AVaR{\norm[L^2(\domain)]{\iota_0 S(u,\cdot)-y_d}^2}
	+ \wp(\norm[L^2(\domain)]{u}),
\end{align}
where $\wp \colon [0,\infty) \to [0,\infty)$, 
$y_d \in L^2(\domain)$,
$\adcsp \subset L^2(\domain)$,
and for each $(u,\xi) \in  L^2(\domain) \times \Xi$,
$y =S(u,\xi) \in H_0^1(\domain)$  solves
the semilinear PDE
\begin{align*}
	\inner[L^2(\domain)^d]{\rdc(\xi)\nabla y}{\nabla v}
	+\inner[L^2(\domain)]{y^3}{v} =
	\inner[L^2(\domain)]{u}{v}+
	\inner[L^2(\domain)]{\rrhs(\xi)}{v}
	\:\: \text{for all} \:\: v \in H_0^1(\domain).
\end{align*}

We impose mild conditions on the semilinear PDE, 
the feasible set $\adcsp$, and  $\wp$.

\begin{assumption}[{semilinear control problem}]
	\label{ass:slocp}
	~
	\begin{enumthm}[leftmargin=*,nosep]
		\item 
		\label{itm:slocp:domain}
		$\domain \subset \real^d$ is a bounded, convex,
		polygonal/polyhedral domain
		with $d \in  \{2,3\}$.
		\item 
		\label{itm:slocp:kappa}
		$\kappa :  \Xi \to L^\infty(\domain)$ is strongly measurable
		and there exists $0 < \kappa_{\min} \leq \kappa_{\max} <\infty$
		such that for all $\xi \in \Xi$,
		$\kappa_{\min} \leq \kappa(\xi)(x) \leq \kappa_{\max}$
		for a.e.\ $x \in \domain$.
		\item 
		\label{itm:slocp:rrhs}
		$\rrhs : \Xi \to L^2(\domain)$
		is measurable and there exists
		$b_{\max} \geq 0$
		with 
		$\norm[L^2(\domain)]{\rrhs(\xi)} \leq b_{\max}$
		for all $\xi \in \Xi$. 
		\item 
		$\adcsp \subset L^2(\domain)$ is closed, convex, nonempty, 
		and there exists $r_{\text{ad}} \in (0,\infty)$ such that
		$\norm[L^2(\domain)]{u} \leq r_{\text{ad}}$
		for all $u \in \adcsp$.
		\item $\wp \colon [0,\infty) \to [0,\infty)$
		is convex and strictly increasing. 
	\end{enumthm}
\end{assumption}

If not stated otherwise, let \Cref{ass:slocp} hold throughout the section. 
Let us define $G 	\colon  \real \times \real \to \real $ 
and $\widehat{J} \colon L^2(\domain) \times \Xi \to [0,\infty)$
by
\begin{align*}
	G(s,t) =
	t + \tfrac{1}{1-\beta}  
	\maxo{s-t}
	\quad \text{and} \quad 
	\widehat{J}(u,\xi) = (1/2) \norm[L^2(\domain)]{\iota_0 S(u,\xi)-y_d}^2.
\end{align*}
Moreover, we define $L^2(\domain)  \times \real
\times \Xi \ni (u,t,\xi) \mapsto
\pobj_\xi(u,t) \in [0,\infty)$ by
\begin{align*}
	\pobj_\xi(u,t)  = G(\widehat{J}(u,\xi),t).
\end{align*}
To study consistency of SAA C-stationary points, we reformulate \eqref{eq:ocpspde}
equivalently as
\begin{align}
	\label{eq:ocpspde'}
	\min_{(u,t) \in \adcsp \times \real}\,
	\int_{\Xi} G(\widehat{J}(u,\xi),t) \, \du \mathbb{P}(\xi)
	+ \wp(\norm[L^2(\domain)]{u}).
\end{align}
Let $\Xi$, let $\mathbb{P}$, 
and let $\xi^1$, $\xi^2, \ldots$  be as in \cref{sec:mcsaa}. 
The SAA problem of \eqref{eq:ocpspde'} is given by
\begin{align}
	\label{eq:saapde'}
	\min_{(u,t) \in \adcsp \times \real}\,
	\frac{1}{N} \sum_{i=1}^N G(\widehat{J}(u,\xi^i),t) 
	+ \wp(\norm[L^2(\domain)]{u}).
\end{align}

For the remainder of the section,
we verify \Cref{assumption:mcsampling}. 
For the spaces $L^2(\domain) \times \real$,
$V = H_0^1(\domain) \times \real$, and
the open, convex, and bounded set
$\csp_0 = \adcsp + \mathbb{B}_{L^2(\domain)}(0;1)$,
\Cref{assumption:mcsampling:spaces} 
holds true.
Let $\psi$ be  the sum of the indicator function of $\adcsp$
and the regularizer $\wp \circ \norm[L^2(\domain)]{\cdot}$.
Since $\wp$ is convex and strictly increasing,
and $L^2(\domain)$ is a Hilbert space, $\psi$ has the Kadec property;
see  \cite[Lem.\ 1]{Milz2022a} and \cite[Thm.\ 6.5]{Borwein1991}. Hence
\Cref{assumption:mcsampling:psi} holds true.

For each $(u,\xi) \in  L^2(\domain) \times \Xi$, 
the PDE solution 
$S(u,\xi)$ is well-defined  \cite[Thm.\ 26.A]{Zeidler1990}, and
we have for all $(u,\xi) \in L^2(\domain) \times \Xi$,
the stability estimate
\begin{align}
	\label{eq:ocpspde_stability}
	\norm[H_0^1(\domain)]{S(u,\xi)} \leq 
	(C_\domain /\kappa_{\min}) \norm[L^2(\domain)]{u}
	+ C_\domain (b_{\max}/\kappa_{\min}).
\end{align}
Defining
\begin{align*}
	\mathfrak{c}_S
	=  (C_\domain /\kappa_{\min}) (r_{\text{ad}}+1)
	+ C_\domain (b_{\max}/\kappa_{\min}),
\end{align*}
we obtain for all 
$(u,\xi) \in \csp_0\times \Xi$, the  stability estimate
\begin{align}
	\label{eq:ocpspde_integrable}
	\widehat{J}(u,\xi) \leq 
	\norm[L^2(\domain)]{\iota_0 S(u,\xi)}^2
	+
	\norm[L^2(\domain)]{y_d}^2
	\leq C_\domain^2\mathfrak{c}_S^2 + \norm[L^2(\domain)]{y_d}^2.
\end{align}

The adjoint approach  \cite[sect.\ 1.6.2]{Hinze2009} ensures
that  for each $\xi \in \Xi$, $\widehat{J}(\cdot,\xi)$ is continuously
differentiable and
\begin{align*}
	\nabla_u \widehat{J}(\cdot,\xi) =  -\iota_0 z(u,\xi)
	\quad \text{for all} \quad (u,\xi) \in L^2(\domain) \times \Xi,
\end{align*}
where for each $ (u,\xi) \in L^2(\domain) \times \Xi$, 
$z = z(u,\xi) \in H_0^1(\domain)$ solves  the adjoint equation
\begin{align*}
	\inner[L^2(\domain)^d]{\rdc(\xi)\nabla z}{\nabla v}
	+3\inner[L^2(\domain)]{S(u,\xi)^2z}{v} =
	-\inner[L^2(\domain)]{\iota_0 S(u,\xi)-y_d}{v}
\end{align*}
for all $v \in H_0^1(\domain)$. For all
$ (u,\xi) \in L^2(\domain) \times \Xi$, we have
\begin{align*}
	\norm[H_0^1(\domain)]{z(u,\xi)} \leq 
	(C_\domain/\kappa_{\min})\norm[L^2(\domain)]{\iota_0 S(u,\xi)}
	+  (C_\domain/\kappa_{\min}) \norm[L^2(\domain)]{y_d}.
\end{align*}
Combined with the stability estimate \eqref{eq:ocpspde_stability},
we have for each $(u,\xi) \in\csp_0\times \Xi$,
\begin{align}
	\label{eq:ocpspde_adstability}
	\norm[H_0^1(\domain)]{z(u,\xi)} \leq 
	(C_\domain^2/\kappa_{\min})\mathfrak{c}_S
	+  (C_\domain/\kappa_{\min}) \norm[L^2(\domain)]{y_d}.
\end{align}
Let $\mathfrak{c}_z$ be the right-hand side in \eqref{eq:ocpspde_adstability}.
Since $G$ is  Lipschitz continuous
and the bound
$\norm[L^2(\domain)]{\nabla_u\widehat{J}(\cdot,\xi)}
\leq C_\domain \mathfrak{c}_z$ is valid
for all $(u,\xi) \in\csp_0\times \Xi$,
\Cref{assumption:mcsampling:integrable} 
is satisfied.

Fix $\xi \in \Xi$. 
Since the mapping 
$(u,t) \mapsto (\widehat{J}(u,\xi),t) \in \real^2$ is continuously
differentiable and $g$ is finite-valued and convex, 
$\pobj_\xi$ is Clarke regular   on $L^2(\domain) \times \real$ 
and
\begin{align*}
	\partial_C \pobj_\xi(u,t)  =
	(\nabla_u \widehat{J}(u,\xi), -1)^*
	\partial G(\widehat{J}(u,\xi),t)
	=
	\begin{bmatrix}
		\frac{1}{1-\beta}\partial \maxo{\widehat{J}(u,\xi)-t} 
		\nabla_u \widehat{J}(u,\xi)
		\\
		1- \frac{1}{1-\beta}\partial \maxo{\widehat{J}(u,\xi)-t} 
	\end{bmatrix};
\end{align*}
see \cite[Prop.\ 2.3.6 and Thm.\ 2.3.10]{Clarke1990}. The Clarke
regularity ensures \Cref{assumption:mcsampling:regular}.
Using  \cite[Thm.\ 8.2.9]{Aubin2009}, we can show that
$S(u,\cdot)$ is  measurable for each $u \in L^2(\domain)$.
Combined with \eqref{eq:ocpspde_integrable}, 
we find that \Cref{assumption:mcsampling:mi} holds true.

We define 
$\coperator \colon H_0^1(\domain) \times \real \to L^2(\domain)\times \real$
and 
$M_\xi \colon L^2(\domain) \times \real \rightrightarrows
H_0^1(\domain) \times \real
$
by
\begin{align}
	\label{eq:sleq_BM}
	\coperator(v,s) = (-\iota_0 v, s)
	\quad \text{and} \quad 
	M_\xi(u,t)
	=
	\begin{bmatrix}
		\frac{1}{1-\beta}\partial \maxo{\widehat{J}(u,\xi)-t} 
		z(u,\xi)
		\\
		1- \frac{1}{1-\beta}\partial \maxo{\widehat{J}(u,\xi)-t} 
	\end{bmatrix}.
\end{align}
We obtain
\begin{align*}
	\partial_C \pobj_\xi(u,t) =
	BM_\xi(u,t)
	\quad \text{for all} \quad (u,t,\xi) \in L^2(\domain) \times \real
	\times \Xi.
\end{align*}
Combined with the fact that $\iota_0$ is linear and compact
\cite[Thm.\ 1.14]{Hinze2009}, 
we conclude that \Cref{assumption:mcsampling:subdifferentialidentify} 
is fulfilled.

Next we verify \Cref{assumption:mcsampling:subdifferentialbounded} 
using \Cref{lem:mcsampling:subdifferentialbounded}.
For each $s \in \real$, $\partial \maxo{s} \subset [0,1]$.
Fix $ (u,t,\xi) \in \adcsp \times \real \times \Xi$.
For all $m = (m_1,m_2)\in M_\xi(u,t)$, 
\eqref{eq:sleq_BM} yields
\begin{align*}
	\norm[H_0^1(\domain) \times \real]{m}
	= \norm[H_0^1(\domain) ]{m_1} +
	|m_2|
	\leq 	\tfrac{1}{1-\beta} \mathfrak{c}_z 
	+1 + \tfrac{1}{1-\beta}.
\end{align*}
Hence \Cref{lem:mcsampling:subdifferentialbounded} implies
\Cref{assumption:mcsampling:subdifferentialbounded}.

To summarize, we have shown that \Cref{assumption:mcsampling} 
holds true.
In order to demonstrate the convergence of 
SAA C-stationary points of \eqref{eq:saapde'} towards 
those of  \eqref{eq:ocpspde'} 
via \Cref{thm:consistencymc}, we need to show that the
SAA C-stationary points are contained in a bounded, deterministic set.
Let $(u_N, t_N) \in \adcsp \times \real$ be  a C-stationary point of
\eqref{eq:ocpspde'}. We have $u_N \in \adcsp$,
which is by assumption a bounded set.  Next we show that
\begin{align*}
	t_N \in [0, C_\domain^2\mathfrak{c}_S^2 + \norm[L^2(\domain)]{y_d}^2].
\end{align*}
Since $(u_N,t_N)$ is a C-stationary point  of
\eqref{eq:ocpspde'},  the second formula in \eqref{eq:sleq_BM} implies that
$1 \in \tfrac{1}{N(1-\beta)} \sum_{i=1}^N
\partial\maxo{\widehat{J}(u_N,\xi^i)-t_N}$. 
Combined with $\beta \in (0,1)$ and $\widehat{J} \geq 0$, 
we have $t_N \geq 0$. Let $\mcAVaR{}$ be the empirical estimate
of $\AVaR{}$ with $\mathbb{P}$ replaced by the empirical 
distribution of the sample $\xi^1, \ldots, \xi^N$ in \eqref{eq:avar}.
The minimization problem in \eqref{eq:avar} with 
$\mathbb{P}$ replaced by the empirical distribution is convex.
Hence $t_N$ solves \eqref{eq:avar} with 
$\mathbb{P}$ replaced by the empirical distribution.
As with $\AVaR{}$ \cite[p.\ 231]{Shapiro2021}, 
the empirical average value-at-risk $\mcAVaR{}$ is a coherent 
risk measure. 
Since $\maxo{s} \geq 0$ for all $s \in \real$, 
$\beta \in (0,1)$, 
and coherent risk measures are monotone,
translation equivariant, and positively homogeneous
\cite[Def.\ 6.4]{Shapiro2021},
the upper bound
on $\widehat{J}$  in \eqref{eq:ocpspde_integrable} yields
\begin{align*}
	t_N \leq \mcAVaR{\widehat{J}(u_N,\xi)} \leq 
	C_\domain^2\mathfrak{c}_S^2 + \norm[L^2(\domain)]{y_d}^2.
\end{align*}

\section{Application to risk-neutral bilinear PDE-constrained optimization}
\label{subsec:blocp}%
Motivated by the deterministic bilinear control problems 
studied in \cite{Casas2018a,Glowinski2022,Kahlbacher2012,Kroener2009},
we consider risk-neutral optimization problems
governed by a bilinear elliptic PDEs with random inputs:
\begin{align}
	\label{eq:blocp}
	\min_{u\in \adcsp}\, 
	\int_{\Xi} \pttobj(S(u,\xi)) \, \du \mathbb{P}(\xi)
	+ (\alpha/2)\norm[L^2(\domain)]{u}^2,
\end{align}
where $\alpha >0$,  and
for each $(u,\xi) \in \adcsp \times \Xi$, 
$y=S(u,\xi) \in H_0^1(\domain)$ is the solution to
\begin{align}
	\label{eq:bleq}
	\inner[L^2(\domain)^d]{\rdc(\xi)\nabla y}{\nabla v}
	+\inner[L^2(\domain)]{g(\xi) u y}{v} =
	\inner[L^2(\domain)]{\rrhs(\xi)}{v}
	\quad \text{for all} \quad v \in H_0^1(\domain).
\end{align}

The following assumptions ensure existence and  regularity
of the solutions to \eqref{eq:bleq},
and impose conditions on the function $\pttobj$ and the set $\adcsp$.
These conditions allow us to \Cref{assumption:mcsampling} 
and apply \Cref{thm:consistencymc}
to study the consistency of SAA C-stationary points corresponding to 
\eqref{eq:blocp}.  Conditions beyond those formulated in 
\Cref{ass:blocp} may be required to allow for applications of 
\Cref{thm:consistencyw}.

\begin{assumption}[{bilinear control problem}]
	\label{ass:blocp}
	~
	\begin{enumthm}[leftmargin=*,nosep]
		\item 
		\label{itm:blocp:domain}
		$\domain \subset \real^d$ is a bounded, convex,
		polygonal/polyhedral domain
		with $d \in  \{2,3\}$.
		\item 
		\label{itm:blocp:kappa}
		$\kappa :  \Xi \to C^1(\bar{\domain})$ is measurable
		and there exists $\kappa_{\min} > 0$
		such that $\kappa_{\min} \leq \kappa(\xi)(x)$
		for all $(\xi, x) \in \Xi \times \bar{\domain}$.
		Moreover
		$\mathbb{E}_{\mathbb{P}}[
		\norm[C^{1}(\bar{\domain})]{\kappa}^p]< \infty$
		for all $p \in [1,\infty)$.
		\item 
		\label{itm:blocp:rrhs}
		$\rrhs : \Xi \to L^2(\domain)$
		and $g : \Xi \to C^{0,1}(\bar{\domain})$
		are measurable and there exist
		$b_{\max}$, $g_{\max} > 0$
		with 
		$\norm[L^2(\domain)]{\rrhs(\xi)} \leq b_{\max}$
		and
		$\norm[C^{0,1}(\bar{\domain})]{g(\xi)} \leq g_{\max}$
		for all $\xi \in \Xi$.
		Moreover, $g(\xi)(x) \geq 0$ for all $(\xi,x) \in \Xi \times \bar{\domain}$.
		
		\item 
		\label{itm:blocp:pobj}
		$\pttobj : H_0^1(\domain) \to [0,\infty)$ is 
		convex, continuously 
		differentiable and its derivative is Lipschitz continuous with
		Lipschitz constant $\ell > 0$.
		For some nondecreasing function $\varrho : [0,\infty) \to [0,\infty)$, 
		it holds that
		\begin{align*}
			\pttobj(y) \leq \varrho(\norm[H_0^1(\domain)]{y})
			\quad \text{for all} \quad y \in H_0^1(\domain).
		\end{align*}
		\item 
		\label{itm:blocp:adcsp}
		$ \adcsp = \{\, u \in L^2(\domain) :\, 0 \leq u(x) \leq \ub(x)
		\; \text{a.e.} \; x \in \domain  \,\}
		$
		with $\ub \in L^\infty(\domain)$ and $\ub(x) \geq 0$
		a.e.\  $x \in \domain$.
	\end{enumthm}
\end{assumption}

If not stated otherwise, let \Cref{ass:blocp} hold throughout the section. 
We  introduce  notation used throughout the section.
Let $C_{H_0^1; L^4} \in (0,\infty)$ be the embedding constant of the
compact embedding $H_0^1(\domain) \embedding L^4(\domain)$
and let $C_{H^2; L^\infty}> 0$ be 
that of the continuous embedding
$H^2(\domain) \embedding L^\infty(\domain)$
\cite[Thm.\ 1.14]{Hinze2009}.
We define $C_\ub = \norm[L^\infty(\domain)]{\ub}$
and $|\domain| = \int_\domain 1 \, \du x$.

Let $\psi$ be  the sum of the indicator function of $\adcsp$
and  $(\alpha/2)\norm[L^2(\domain)]{u}^2$.
The function $\psi$ has the Kadec property \cite[Thm.\ 6.5]{Borwein1991},
and \Cref{assumption:mcsampling:psi} holds true.

Next we provide an example function satisfying  \Cref{itm:blocp:pobj}.  
Let  $y_d \in L^2(\domain)$.
We define
$\pttobj: H_0^1(\domain) \to [0,\infty)$ by
$
\pttobj(y) = (1/2)\norm[L^2(\domain)]{\iota_0 y - y_d}^2
$.
The function $\pttobj$ satisfies \Cref{itm:blocp:pobj}
with $\ell = C_\domain^2$
and 
$\varrho(t) = C_\domain^2 t^2 + \norm[L^2(\domain)]{y_d}^2$.

We use two technical facts for our analysis. 
\Cref{lem:selfbounded} establishes a bound on the derivatives
of the objective functions fulfilling \Cref{itm:blocp:pobj}.
\begin{lemma}
	\label{lem:selfbounded}
	If $\domain \subset \real^d$
	is a bounded domain and  \Cref{itm:blocp:pobj} holds, then
	$\norm[H^{-1}(\domain)]{\Du \pttobj(y)}^2 
	\leq 2 \ell \pttobj(y)
	$
	for all $y \in H_0^1(\domain)$.
\end{lemma}
\begin{proof}
	Fix $\tilde y$, $y \in H_0^1(\domain)$.
	Let $\nabla \pttobj(y) \in H_0^1(\domain)$ 
	be the Riesz representation of $\Du \pttobj(y)$.
	Using \Cref{itm:blocp:pobj} and 	
	\cite[Cor.\ 3.5.7 and Rem.\ 3.5.2]{Zalinescu2002}, we have
	$
	0 \leq \pttobj(\tilde y) \leq 
	\pttobj(y) +
	\inner[H_0^1(\domain)]{\nabla \pttobj(y)}{\tilde y-y}
	+ (\ell/2)\norm[H_0^1(\domain)]{\tilde y-y}^2
	$.
	Minimizing over $\tilde y \in H_0^1(\domain)$ yields
	$0 \leq \pttobj(y)
	- (1/(2\ell))\norm[H_0^1(\domain)]{\nabla \pttobj(y)}^2$.
\end{proof}

Using \Cref{itm:blocp:domain}
and theorems on the multiplication of Sobolev functions
(see \cite[Thm.\ 7.4]{Behzadan2021} and \cite[Thm.\ 1.4.1.1]{Grisvard2011}),
we can show that the trilinear (pointwise multiplication)  operator
\begin{align}
	\label{eq:multiplication_operator}
	C^{0,1}(\bar{\domain}) \times H^1(\domain) \times H^2(\domain) 
	\ni (g, v, w) 
	\mapsto gvw \in H^1(\domain)
\end{align} 
is  continuous. Hence  it is bounded  \cite[p.\ 68]{Lang1993}.
Therefore, we obtain the existence of a constant 
$C_{\mathrm{tri}} \in (0,\infty)$ such that
for all $(g, v, w) \in 
C^{0,1}(\bar{\domain}) \times H^1(\domain) \times H^2(\domain)$,
\begin{align}
	\label{eq:multiplication_Grisvard2011}
	\norm[H^1(\domain)]{gvw} \leq C_{\mathrm{tri}} 
	\norm[C^{0,1}(\bar{\domain})]{g}
	\norm[H^1(\domain)]{v}
	\norm[H^2(\domain)]{w}.
\end{align}

\subsection{Existence and uniqueness of PDE solutions}
\label{subsec:bilinear:existence}
We show that the PDE \eqref{eq:bleq} has a unique solution 
$S(u,\xi) \in H_0^1(\domain)$
for each $\xi \in  \Xi$ and $u$ in an open neighborhood
of  the set $\adcsp$ in $L^2(\domain)$. 
This allows us to verify 
\Cref{assumption:mcsampling:integrable,assumption:mcsampling:regular,%
	assumption:mcsampling:mi,assumption:mcsampling:spaces}. 
The existence results established in \cite[sect.\ 3.1]{Casas2018a}
are not applicable to our setting, as we consider $L^2(\domain)$
instead of $L^\infty(\domain)$  the control space. 
An approach different from ours to constructing such a
neighborhood is developed in \cite[p.\ 158]{Winkler2020}.

We define 
\begin{align}
	\label{eq:blinear:delta}
	\csp_0 = \bigcup_{u_0\in \adcsp} \, \mathbb{B}_{L^2(\domain)}(u_0;\delta),
	\quad \text{where} \quad 
	\delta = \tfrac{\kappa_{\min}}{2g_{\max}C_{H_0^1; L^4}^2} > 0.
\end{align}
We have  the identity
$\csp_0 = \adcsp + \mathbb{B}_{L^2(\domain)}(u_0;\delta)$.

Fix $(u,\xi) \in \csp_0 \times \Xi$ and $y \in H_0^1(\domain)$.
Since $u \in \csp_0$, there exists $u_0 \in \adcsp$
with $u \in \mathbb{B}_{L^2(\domain)}(u_0;\delta)$. 
Combining \hoelder's inequality and the continuity
of $H_0^1(\domain) \embedding L^4(\domain)$
with the definition of $\delta$ provided in \eqref{eq:blinear:delta}
and $\norm[L^2(\domain)]{u-u_0} \leq \delta$, we obtain
\begin{align*}
	|\inner[L^2(\domain)]{g(\xi) (u-u_0) y}{y}|
	\leq 
	C_{H_0^1; L^4}^2 g_{\max} \norm[L^2(\domain)]{u-u_0}
	\norm[H_0^1(\domain)]{y}^2
	\leq (\kappa_{\min}/2)\norm[H_0^1(\domain)]{y}^2.
\end{align*}
Using $u_0(x) \geq 0$ and $g(\xi)(x) \geq 0$ for a.e.\ $x \in \domain$,
$\inner[L^2(\domain)]{g(\xi) u_0 y}{y} \geq 0$.
It follows that
\begin{align}
	\label{eq:bilinear:zeroterm}
	\begin{aligned}
		& \inner[L^2(\domain)^d]{\rdc(\xi)\nabla y}{\nabla y}
		+\inner[L^2(\domain)]{g(\xi) u y}{y}
		\\
		& \quad = 
		\inner[L^2(\domain)^d]{\rdc(\xi)\nabla y}{\nabla y}
		+\inner[L^2(\domain)]{g(\xi) (u-u_0) y}{y}
		+ \inner[L^2(\domain)]{g(\xi) u_0 y}{y}
		\\
		& \quad  
		\geq \rdc_{\min}\norm[H_0^1(\domain)]{y}^2 -
		(\rdc_{\min}/2)\norm[H_0^1(\domain)]{y}^2 = 
		(\rdc_{\min}/2)\norm[H_0^1(\domain)]{y}^2.
	\end{aligned}
\end{align}
The right-hand side in
\eqref{eq:bleq} defines a continuous bilinear form.
Hence the Lax--Milgram lemma ensures that the
state equation 
\eqref{eq:bleq} has a unique solution $S(u,\xi) \in H_0^1(\domain)$
and yields the stability estimate
\begin{align}
	\label{eq:bilinear:H1SVadcsp}
	\norm[H_0^1(\domain)]{S(u,\xi)} \leq 
	(2/\kappa_{\min}) C_\domain \norm[L^2(\domain)]{\rrhs(\xi)}
	\quad \text{for all} \quad (u,\xi) \in \csp_0 \times \Xi.
\end{align}

We define the parameterized operator 
$E : H_0^1(\domain) \times L^2(\domain) \times \Xi \to H^{-1}(\domain)$ by
\begin{align*}
	\dualpHzeroone{E(y,u,\xi)}{v}
	= \inner[L^2(\domain)^d]{\rdc(\xi)\nabla y}{\nabla v}
	+\inner[L^2(\domain)]{g(\xi) u y}{v} -
	\inner[L^2(\domain)]{\rrhs(\xi)}{v}.
\end{align*}
The mapping
$E$ is well-defined
and $E(\cdot,\cdot,\xi)$ is infinitely many times continuously
differentiable. 
For each $(y,u,\xi) \in H_0^1(\domain) \times L^2(\domain) \times \Xi$
and $h$, $v \in H_0^1(\domain)$, we have
\begin{align*}
	\begin{aligned}
		\dualpHzeroone{E_y(y,u,\xi)h}{v}
		& = \inner[L^2(\domain)^d]{\rdc(\xi)\nabla h}{\nabla v}
		+\inner[L^2(\domain)]{g(\xi) u h}{v}.
	\end{aligned}
\end{align*}
Together with  \eqref{eq:bilinear:zeroterm}, we obtain
for all $(u,\xi) \in\csp_0 \times \Xi$,
\begin{align*}
	\dualpHzeroone{E_y(S(u,\xi),u,\xi)h}{h}
	\geq (\kappa_{\min}/2) \norm[H_0^1(\domain)]{h}^2
	\quad \text{for all} \quad h \in H_0^1(\domain).
\end{align*}
The Lax--Milgram lemma ensures that 
$E_y(S(u,\xi),u,\xi)$ has a bounded inverse and
\begin{align}
	\label{eq:bilinear:EyboundVadcsp}
	\norm[\spL{H_0^1(\domain)}{H^{-1}(\domain)}]{E_y(S(u,\xi),u,\xi)^{-1}}
	\leq (2/\kappa_{\min})
	\quad \text{for all} \quad (u,\xi) \in \csp_0 \times \Xi.
\end{align}
For each $y$,  $v \in H_0^1(\domain)$
and $u \in L^2(\domain)$, we can show that
$\dualpHzeroone{E(y,u,\cdot)}{v}$ is measurable,
as \Cref{itm:blocp:domain,itm:blocp:kappa,itm:blocp:rrhs} hold.
Since $H^{-1}(\domain)$ is separable, 
$E(y,u,\cdot)$ is measurable \cite[Thm.\ 1.1.6]{Hytoenen2016}.
Hence $S(u,  \cdot)$ is  measurable \cite[Thm.\ 8.2.9]{Aubin2009}.
Combined with \Cref{itm:blocp:pobj}, we find that
$\pttobj(S(u,\cdot))$ is a random variable for each $u\in\csp_0$.

\subsection{PDE regularity}
\label{subsec:bilinear:higherregularity}

We show that the solution $S(u,\xi)$ to 
\eqref{eq:bleq} is contained in $H^2(\domain)$
and derive a stability estimate for 
each $(u,\xi) \in \adcsp \times \Xi$.
The computations performed here form the basis of verifying
\Cref{assumption:mcsampling:subdifferentialidentify}.

Fix  $(u,\xi) \in \adcsp \times \Xi$. Hence $u \in L^\infty(\domain)$.
To establish an $H^2(\domain)$-stability estimate, we 
apply \cite[Thm.\ 3.1]{Teckentrup2013} to the solution
$\widetilde{y} = \widetilde{y}(u,\xi) \in H_0^1(\domain)$ to
\begin{align*}
	\inner[L^2(\domain)^d]{\rdc(\xi)\nabla \widetilde{y}}{\nabla v}
	=
	\inner[L^2(\domain)]{\rrhs(\xi)}{v}-
	\inner[L^2(\domain)]{g(\xi) u S(u,\xi)}{v}
	\quad \text{for all} \quad v \in H_0^1(\domain).
\end{align*}
Since this equation has a unique solution $\widetilde{y}(u,\xi)$ 
and $S(u,\xi)$ is a solution, we have $\widetilde{y}(u,\xi) = S(u,\xi)$.
Defining the random variable
\begin{align}
	\label{eq:Ckappa}
	C_\kappa(\xi) = \frac{\norm[C^0(\bar{\domain})]{\kappa(\xi)}
		\norm[C^1(\bar{\domain})]{\kappa(\xi)}^2}
	{\kappa_{\min}^4},
\end{align}
and applying \cite[Thm.\ 2.1]{Teckentrup2013}, we find that
$S(u,\xi) \in H^2(\domain)$ and
\begin{align*}
	\norm[H^2(\domain)]{S(u,\xi)}
	\leq 
	C_{H^2}C_\kappa(\xi)						
	\Big(			\norm[L^2(\domain)]{g(\xi)uS(u,\xi)}
	+
	\norm[L^2(\domain)]{\rrhs(\xi)}
	\Big),
\end{align*}
where $C_{H^2} = C_{H^2}(\domain) > 0$ is a deterministic constant.
Combined with \hoelder's and \friedrichs\' inequalities,
we obtain for all $(u,\xi) \in \adcsp \times \Xi$,
\begin{align}
	\label{eq:bleq_h2_state}
	\norm[H^2(\domain)]{S(u,\xi)}
	\leq
	C_{H^2}C_\kappa(\xi)					
	\big(		
	C_D	C_\ub g_{\max} 
	\norm[H_0^1(\domain)]{S(u,\xi)}
	+
	\norm[L^2(\domain)]{\rrhs(\xi)}
	\big).
\end{align}

\subsection{Gradient regularity}
\label{subsec:gradientcomputation}
In this section, we show that 
\Cref{assumption:mcsampling:integrable,%
	assumption:mcsampling:regular,%
	assumption:mcsampling:subdifferentialidentify,%
	assumption:mcsampling:subdifferentialbounded}
hold true.
Using the construction of $\csp_0$,  we can define the objective function
$\csp_0 \times \Xi \ni (u,\xi) \mapsto \rpobj_\xi(u) \in [0,\infty)$ by
\begin{align*}
	\rpobj_\xi(u) = \pttobj(S(u,\xi)).
\end{align*}

\Cref{ass:blocp}
ensures that 
$\rpobj_{\xi}$ is continuously differentiable for every $\xi \in \Xi$. Hence
\Cref{assumption:mcsampling:regular} holds true 
\cite[Prop.\ 2.3.6]{Clarke1990}.
The adjoint approach  \cite[sect.\ 1.6.2]{Hinze2009} yields
for each $(u,\xi) \in \csp_0 \times \Xi$,
\begin{align}
	\label{eq:bilinear:gradient}
	\inner[L^2(\domain)]{\nabla \rpobj_{\xi}(u)}{s}
	= \inner[L^2(\domain)]{g(\xi)S(u,\xi)z(u,\xi)}{s}
	\quad \text{for all} \quad s \in L^2(\domain),
\end{align}
where for each $(u,\xi) \in \csp_0 \times \Xi$, the adjoint state
$z = z(u,\xi) \in H_0^1(\domain)$ solves
\begin{align}
	\label{eq:bloc:adjointeq}
	\inner[L^2(\domain)^d]{\rdc(\xi)\nabla z}{\nabla v}
	+\inner[L^2(\domain)]{g(\xi) u z}{v} =
	-\dualpHzeroone{\Du \pttobj(S(u,\xi))}{v}
\end{align}
for all $v \in H_0^1(\domain)$.
Using \eqref{eq:bilinear:EyboundVadcsp} and \eqref{eq:bloc:adjointeq}, 
we obtain the stability estimate
\begin{align*}
	\begin{aligned}
		\norm[H_0^1(\domain)]{z(u,\xi)} 
		&\leq (2/\kappa_{\min}) 
		\norm[H^{-1}(\domain)]{\Du \pttobj(S(u,\xi))}
		\;\; \text{for all} \;\;
		(u,\xi) \in \csp_0 \times \Xi.
	\end{aligned}
\end{align*}
Combined with the stability estimate \eqref{eq:bilinear:H1SVadcsp}, 
\Cref{itm:blocp:pobj}, and \Cref{lem:selfbounded}, 
we have for all $(u,\xi) \in \csp_0 \times \Xi$,
\begin{align}
	\label{eq:bleq_h2_adjoint}
	\begin{aligned}
		\norm[H_0^1(\domain)]{z(u,\xi)} 
		&\leq (2/\kappa_{\min}) 
		\big(2\ell \varrho((2/\kappa_{\min}) C_\domain \rrhs_{\max})\big)^{1/2}.
	\end{aligned}
\end{align}

We show that \Cref{assumption:mcsampling:integrable} 
is satisfied.
Using the \hoelder\ and  \friedrichs\  inequalities, the continuity
of $H_0^1(\domain) \embedding L^4(\domain)$, the stability estimates
\eqref{eq:bilinear:H1SVadcsp} and \eqref{eq:bleq_h2_adjoint}, 
and the gradient formula \eqref{eq:bilinear:gradient}, we have
$(u,\xi) \in \csp_0 \times \Xi$,
\begin{align*}
	\norm[L^2(\domain)]{\nabla \rpobj_{\xi}(u)}
	& \leq C_{H_0^1; L^4}^2 g_{\max}
	\norm[H_0^1(\domain)]{S(u,\xi)}
	\norm[H_0^1(\domain)]{z(u,\xi)}
	\\
	& \leq 
	C_{H_0^1; L^4}^2 g_{\max}
	(2/\kappa_{\min})^2 C_\domain \rrhs_{\max}
	\big(2\ell \varrho((2/\kappa_{\min}) C_\domain \rrhs_{\max})\big)^{1/2}.
\end{align*}
Right-hand side in 
the above equation is a deterministic 
constant. Hence \Cref{assumption:mcsampling:integrable}  holds true.

The stability estimate
\eqref{eq:bilinear:H1SVadcsp} and \Cref{itm:blocp:pobj} yield
for all $(u,\xi) \in \csp_0 \times \Xi$,
\begin{align}
	\label{eq:bilinear:rpobjbound}
	\rpobj_{\xi}(u)
	\leq \varrho((2/\kappa_{\min}) C_\domain \rrhs_{\max}).
\end{align}
Combined with the measurability of
$\pttobj(S(u,\cdot))$  for each $u\in\csp_0$, 
we find that \Cref{assumption:mcsampling:mi} holds true.

We verify \Cref{assumption:mcsampling:subdifferentialidentify}
by showing that for each $(u,\xi) \in \adcsp \times \Xi$,
\begin{align}
	\label{eq:2021-06-09T17:00:47.21}
	\nabla \rpobj_{\xi}(u) = \iota_1 m_\xi(u),
	\quad \text{where} \quad   m_\xi(u) =   g(\xi)S(u,\xi)z(u,\xi).
\end{align}
We define $M_\xi(u) = \{\, m_\xi(u)\, \}$, $V = H^1(\domain)$, 
and $\coperator = \iota_1$.
Fix $(u,\xi) \in \adcsp \times \Xi$. 
We have $S(u,\xi) \in   H^2(\domain)$
(see \cref{subsec:bilinear:higherregularity} and \eqref{eq:bleq_h2_state})
and $z(u,\xi) \in H^1(\domain)$ (see \eqref{eq:bleq_h2_adjoint}).
Combined with $g(\xi) \in C^{0,1}(\bar{\domain})$
and \eqref{eq:multiplication_Grisvard2011},
we find that $m_\xi(u)\in H^1(\domain)$.
Hence  \eqref{eq:bilinear:gradient} implies
\eqref{eq:2021-06-09T17:00:47.21}.

We show that \Cref{assumption:mcsampling:subdifferentialbounded} 
holds true using \Cref{lem:mcsampling:subdifferentialbounded}. 
Fix $(u,\xi) \in \adcsp \times \Xi$. Using 
\eqref{eq:multiplication_Grisvard2011}
and \eqref{eq:2021-06-09T17:00:47.21}, we have
\begin{align*}
	\norm[H^1(\domain)]{m_\xi(u)}
	\leq 
	C_{\mathrm{tri}} g_{\max}
	\norm[H^2(\domain)]{S(u,\xi)} \norm[H^1(\domain)]{z(u,\xi)}.
\end{align*}
Combined with the stability estimates 
\eqref{eq:bilinear:H1SVadcsp},
\eqref{eq:bleq_h2_state}, and
\eqref{eq:bleq_h2_adjoint}, and 
$\norm[H^1(\domain)]{y} \leq (1+C_\domain) \norm[H_0^1(\domain)]{y}$
valid for all $y \in H_0^1(\domain)$, 
we find that
\begin{align*}
	\norm[H^1(\domain)]{m_\xi(u)}
	& \leq C_{\mathrm{tri}} g_{\max}
	C_{H^2}C_\kappa(\xi)					
	\big(		
	C_D^2	C_\ub g_{\max}
	(2/\kappa_{\min})  \rrhs_{\max}
	+
	\rrhs_{\max}
	\big)
	\\
	& \quad \cdot  (2/\kappa_{\min}) C_\domain (1+C_\domain)
	\big(2\ell \varrho((2/\kappa_{\min}) C_\domain \rrhs_{\max})\big)^{1/2}.
\end{align*}
We define $\zeta$ as the right-hand side in the above equation.
Using \Cref{itm:blocp:kappa} and 
\eqref{eq:Ckappa}, we find that $\zeta$ is integrable.
Hence \Cref{lem:mcsampling:subdifferentialbounded} yields
\Cref{assumption:mcsampling:subdifferentialbounded}.

\section{Discussion}
\label{sec:conclusion}
A multitude of applications in science and engineering 
yield  PDE-constrained optimization problems
under uncertainty. Motivated by such applications, we have considered 
infinite-dimensional stochastic optimization problems. We approximated
the expectations in the stochastic programs using two approaches:
the SAA approach and using probability measures weakly converging
to the random element's probability
distribution. For both approximation approaches, we established
asymptotic consistency statements for C-stationary points
under Clarke regularity of integrands. We applied our framework
to a risk-averse semilinear PDE-constrained optimization problem
with the average value-at-risk as a risk measure,
and to risk-neutral bilinear PDE-constrained optimization.

Our consistency analysis requires integrands be Clarke regular.
For many applications in PDE-constrained optimization, integrands
are continuously differentiable and hence Clarke regular. 
For nonregular integrands, consistency analysis of SAA 
first-order stationary points is more demanding. 
Using  smoothing functions of integrands, 
the SAA approach has been analyzed in \cite{Burke2020,Xu2009}
as the sample size approaches infinity and smoothing parameters
converge to zero.
We refer the reader to \cite{Qi2022} for recent contributions to the
consistency analysis of SAA stationary points as applied to risk-neutral
optimization problems with  integrands that lack Clarke regularity.

Our main result, \Cref{thm:consistencymc}, 
establishes the consistency of 
SAA C-stationary points of infinite-dimensional risk-neutral optimization problems. 
While applicable to a large set of applications, 
it may be desirable to establish consistency results for 
sample-based approximations of
chance-constrained problems governed by PDEs with random inputs
\cite{Chen2020,FarshbafShaker2020,Tong2022}, and
PDE-constrained problems under uncertainty with state constraints
\cite{Geiersbach2021b,Kouri2023}. Further future
work includes extending  the consistency analysis
to a broader class of 
risk-averse PDE-constrained optimization problems.

\section*{Acknowledgments}
JM is very grateful to 
Professor Alexander Shapiro for taking the time to 
address several of my questions related to the SAA approach.

{
\footnotesize
\bibliography{JMilz_bilinear_epi.bbl}
}

\end{document}